\documentclass[12pt]{article}

\usepackage{amsmath,amsthm,amsfonts,amssymb,amscd}

\usepackage{fullpage}

\allowdisplaybreaks[4]

\newtheorem {Definition}{Definition}[section]
\newtheorem {Lemma}{Lemma}[section]
\newtheorem {Theorem} {Theorem}[section]

\newtheorem {Example} {Example}[section]

\newtheorem{Proposition}{Proposition}[section]
\newtheorem {Corollary}{Corollary}[section]

\numberwithin{equation}{section}

\begin{document}

\title{Largest and Least H-Eigenvalues of Symmetric Tensors and Hypergraphs}

\author{Hongying Lin$^{a}$\footnote{E-mail: linhy99@scut.edu.cn},
 Lu Zheng$^{b}$\footnote{E-mail: zhenglu@m.scnu.edu.cn}, Bo Zhou$^{b}$\footnote{Corresponding author. E-mail: zhoubo@scnu.edu.cn}\\
$^{a}$School of Mathematics, South China  University of Technology, \\
Guangzhou 510641, P.R. China \\
$^{b}$School of  Mathematical Sciences, South China Normal University, \\
 Guangzhou 510631, P.R. China
}

\date{}
\maketitle

\begin{abstract}
In tensor eigenvalue problems, one is likely to be more interested in H-eigenvalues of tensors.
The largest H-eigenvalue of a nonnegative tensor or of a uniform hypergraph is the spectral radius of the tensor or of the uniform hypergraph.
We find upper bounds and lower bounds (interlacing inequalities) for the largest H-eigenvalue of a principal subtensor of a symmetric zero diagonal tensor that is  of even order or nonnegative, as well as  lower bounds for the largest H-eigenvalue  of a uniform hypergraph with some vertices or edges  removed. We also investigate similar problems for the least H-eigenvalues.
We  give examples  to verify the sharpness of the bounds or in some cases for uniform hypergraphs, we characterize the equality. Particularly, for a connected linear $k$-uniform hypergraph $G$ with $v\in V(G)$, we give a sharp lower  bound for the spectral radius of  $G-v$ in terms of  the spectral radius of $G$ and the degree of $v$
and characterize the extremal hypergraphs, and show that
the maximum spectral radius of the subhypergraphs with one vertex removed  is greater than or equal to the spectral radius of the hypergraph minus one, which is attained if and only if it is a Steiner system $S(2,k,n)$.
\end{abstract}

\section{Introduction}

Let $\mathbb{R}$ be the field of real numbers and $\mathbb{R}^n$ the $n$-dimensional real space.
For positive integers $k$ and $n$, a (real) tensor (or hypermatrix) $\mathcal{T}=(t_{i_1\ldots i_k})$ of order $k$ and dimension $n$ is a multidimensional array
with entries $t_{i_1\dots i_k}\in \mathbb{R}$ for $i_j \in [n]:=\{1, \dots, n\}$ and  $j\in[k]$.
An entry $t_{i_1\ldots i_k}$ with $i_1=\dots=i_k=i\in[n]$ is a diagonal entry of  $\mathcal{T}$.
A zero diagonal tensor is a tensor for which all diagonal entries are equal to zero.
The tensor $\mathcal{T}$ is symmetric  if each entry $t_{i_1\dots i_k}$
is invariant with respect to  all permutations of $i_1, \dots, i_k$.
A tensor  is nonnegative  if all its entries are nonnegative.

For  a tensor $\mathcal{T}$ of order $k$ and dimension $n$, and  an $n$-dimensional vector  ${\bf x}=(x_1,\ldots, x_n)^\top$,
$\mathcal{T}{\bf x}^{k-1}$ is defined as an $n$-dimensional vector whose $i$-th entry is
\[
(\mathcal{T}{\bf x}^{k-1})_i\equiv
\sum_{i_2,\dots, i_k\in[n]}t_{ii_2\dots i_k}x_{i_2}\cdots x_{i_k}
\]
for $i\in [n]$, and
$\mathcal{T}{\bf x}^{k}$ is defined as the $k$-th degree  homogeneous polynomial
\[
\mathcal{T}{\bf x}^{k}\equiv \sum_{i_1,\dots, i_k\in[n]} t_{i_1\dots i_k}x_{i_1}\cdots x_{i_k}.
\]

\begin{Definition} \cite{Lim, Qi05}  Let $\mathcal{T}$ be a tensor  of order $k$ and dimension $n$.
For some complex $\lambda$, if there is a nonzero vector ${\bf x}$ such that
\[
\lambda x_i^{k-1}=(\mathcal{T}{\bf x}^{k-1})_i,
\]
i.e.,
\begin{align}
\lambda x_i^{k-1}=\sum_{i_2,\dots, i_k\in [n]} t_{ii_2\ldots i_k}x_{i_2}\cdots x_{i_k} \label{eq22-6-22}
\end{align}
for $i\in[n]$,
then $\lambda$ is called an eigenvalue of $\mathcal{T}$, and ${\bf x}$ is called an eigenvector of $\mathcal{T}$ corresponding to $\lambda$. Moreover, if both $\lambda$ and ${\bf x}$ are real, then we call $\lambda$ an H-eigenvalue and ${\bf x}$ an H-eigenvector of $\mathcal{T}$.
\end{Definition}

For more details on tensor eigenvalues and eigenvectors, we refer the readers to \cite{CPT,Ko,QL}.
Let $\mathcal{T}$ be a tensor  of order $k$ and dimension $n$. The spectral radius of $\mathcal{T}$ is the largest modulus of the eigenvalues of $\mathcal{T}$, denoted by  $\rho(\mathcal{T})$.
Suppose that there exists at least one H-eigenvalue of $\mathcal{T}$. For instance, $\mathcal{T}$ has at least one H-eigenvalue if $\mathcal{T}$ is symmetric and $k$ is even, see \cite{Qi05}, or if $\mathcal{T}$ is  nonnegative, see Proposition \ref{PF} below. In this case, we
 denote $\lambda_{\max}(\mathcal{T})$ and $\lambda_{\min}(\mathcal{T})$ the largest H-eigenvalue
and the least H-eigenvalue of $\mathcal{T}$, respectively. In this case,
it is evident that  $\lambda_{\min}(\mathcal{T})\le \lambda_{\max}(\mathcal{T})\le \rho(\mathcal{T})$, and if $\mathcal{T}$ is nonnegative, then $\lambda_{\max}(\mathcal{T})=\rho(\mathcal{T})$.
In most cases, one is likely to be more interested in H-eigenvalues of tensors. 

\begin{Definition} \cite{Fri}
A tensor $\mathcal{T}=(t_{i_i\dots i_k})$ of order $k$ and dimension $n$ is said to be
weakly reducible if
$t_{i_1\dots i_k}=0$ for some $\emptyset \ne I\subset [n]$ and
for any $i_1\in I$ and at least one $j\in \{2,\dots, k\}$ with $i_j\not\in I$. Otherwise, it is weakly irreducible.
\end{Definition}

The classic Perron-Frobenius theorem has been extended to  nonnegative tensors
by the efforts scholars as follows, see \cite{CQT,Fri,YY,YY2} with a unifying  treatment in \cite{GTHM}.

\begin{Proposition} \label{PF}  \cite{Fri}
For a nonnegative tensor $\mathcal{T}$  of order $k$ and dimension $n$ with $n,k\ge 2$,
$\rho(\mathcal{T})$ is an H-eigenvalue of $\mathcal{T}$ with a
positive H-eigenvector. If $\mathcal{T}$ is weakly irreducible, then there is a unique positive H-eigenvector, up to a multiplicative constant, and moreover,
if $\lambda$ is an H-eigenvalue with a positive eigenvector, then $\lambda=\rho(G)$.
\end{Proposition}


Let $G$ be a $k$-uniform hypergraph with vertex set $V(G)=[n]$ and edge set $E(G)$, where $n,k\geq 2$.
For $u\in V(G)$, denote $E_{u}(G)$ be the set of edges containing $u$ in $G$. The degree of $u$ in $G$ is defined as $d_G(u)=|E_u(G)|$, we also write $d_u$ for $d_G(u)$ if there is no confuse.
 For any two distinct vertices $i$ and $j$ of $G$,
we write $i\sim j $ if there is an edge containing $i$ and $j$,
and  $i\nsim j $ otherwise. A linear hypergraph is one in which every two distinct edges intersect in at most one vertex.

\begin{Definition}  \cite{CC,CD}
The adjacency tensor of $G$  is defined as the  symmetric, nonnegative tensor $\mathcal{A}(G)=(a_{i_1 \dots i_k})$ of order $k$ and dimension $n$, where
\[
a_{i_1 \dots i_k}= \begin{cases}
\frac{1}{(k-1)!} & \text{if } \{i_1, \dots, i_k\}\in E(G),\\
0  & \text{otherwise.}
\end{cases}
\]
\end{Definition}

Note that the adjacency tensor of a uniform hypergraph is  nonnegative, so it has at least one H-eigenvalue by Proposition \ref{PF}.

\begin{Definition}
The  spectral radius (or largest H-eigenvalue)  of $\mathcal{A}(G)$ is called the  spectral radius (or the largest H-eigenvalue) of $G$, denoted by
$\rho(G)$, and the least H-eigenvalue of $\mathcal{A}(G)$ called the least H-eigenvalue of $G$, denoted by $\lambda(G)$. That is, $\rho(G)=\rho(\mathcal{A}(G))=\lambda_{\max}(\mathcal{A}(G))$ and $\lambda(G)=\lambda_{\min}(\mathcal{A}(G))$.
\end{Definition}

If $G$ is an ordinary graph, then $\rho(G)$ and $\lambda(G)$ ere respectively the largest and the least eigenvalues of (the adjacency matrix) of $G$ \cite{AS,CRS,Nik1}.

For a $k$-uniform uniform hypergraph $G$, by Proposition \ref{PF}, $\rho(G)$ is an H-eigenvalue of $\mathcal{A}(G)$ with an associated nonnegative H-eigenvector, and moreover,  if $G$ is connected, then
$\mathcal{A}(G)$  is weakly irreducible \cite{PT}, implying that
there is a unique unit  positive H-eigenvector corresponding to  $\rho(G)$. In this article, we say a vector ${\bf x}\in \mathbb{R}^n$ is unit if $\|{\bf x}\|_k^k:=\sum_{i\in [n]}|x_i|^k=1$.
The approach to study of hypergraphs through tensors has been widely accepted, see, e.g., \cite{Ben,CC,CD, ELW,KLS,Kh, LKY,Nik,PT,SSW}. It should be pointed that other treatment of the spectral property of  hypergraphs may be found, see, e,g., \cite{KLM}.

There are many results on the bounds for the eigenvalues (particularly the largest one) of modified graphs by Rowlinson and coauthors, see, e.g.  \cite{BR,CRS,Row}, where a modified graph is obtained from some given graph under small changes such as by removing vertices or edges and moving certain edges.
Li, Wang and Van Mieghem \cite{LWM} presented a new novel type lower bound for the spectral radius of a graph when some vertices are removed.  Van Mieghem et al. \cite{Van} gave bounds for the spectral radius of a graph when some edges are removed. In \cite{XZ},
an upper bound was established  for the least eigenvalue of a graph when some vertices are removed.

\begin{Definition} \cite{KS,Qi05}
For a  tensor $\mathcal{T}=(t_{i_1\ldots i_k})$ of order $k$ and dimension $n$,
a principal subtensor $\mathcal{T}[I]$ of  $\mathcal{T}$ with nonempty index set $I\subseteq [n]$  is a tensor of order $k$ and dimension $|I|$ consisting
of $|I|^k$ elements defined by
\[
\mathcal{T}[I]=(t_{i_1\ldots i_k}) \mbox{ with  $i_1,\dots, i_k\in I$.}
\]
\end{Definition}

\begin{Definition} \label{ST}  \cite{vanW}
A Steiner system $S(t,k,n)$, of order
$n$ and block size $k$ with $n\ge k\ge 2$,  is a collection of $k$-sets  of an $n$-set such that every $t$-set belongs to exactly one block. In other words, $S(t,k,n)$ is a $k$-uniform hypergraph on $n$ vertices, such that every
$t$-element vertex subset is contained in precisely one edge.
\end{Definition}

In this paper, we find upper bounds and lower bounds (interlacing inequalities) for the largest H-eigenvalue of a principal subtensor of a symmetric zero diagonal tensor that is  of even order or nonnegative, from which we derive new
lower bounds for the largest H-eigenvalue (spectral radius) of a uniform hypergraph in which some vertices or edges are removed.
On the other hand, we present upper bounds and lower bounds for the least H-eigenvalue of a principal subtensor of a symmetric zero diagonal tensor of even order, from which we derive new
upper and lower bounds for the least H-eigenvalue of a uniform hypergraph in which some vertices or edges are removed.
We also present some bounds of the components of the least eigenvectors of hypergraphs.
Some results from \cite{KLS, LWM,Nik,Se,Van,XZ} are generalized or improved. Particularly, for  any connected linear $k$-uniform hypergraph $G$ on $n$ vertices with $n\ge k\ge 2$, we give a sharp upper bound on $\rho(G-v)$ with $v\in V(G)$ in terms of $\rho(G)$ and $d_G(v)$ and characterize the hypergraphs for which this bound is attained, and show that $\max\{\rho(G-v): v\in V(G)\}\ge \rho(G)-1$ with equality if and only if $G$ is a Steiner system $S(2,k,n)$.

To the best of our knowledge, there is no such type of lower (upper, respectively)  bounds for the largest (least, respectively)  H-eigenvalues of symmetric tensors  and uniform hypergraphs in the literature. We also give examples  to verify the sharpness of the bounds or in some cases for hypergraphs, we characterize the equality.

\section{Preliminaries}

We now give some tools that will be use later.

\begin{Lemma}\label{sym} \cite[Theorem~5]{Qi05}
Let $\mathcal{T}$ be a symmetric  tensor of even order $k$ and dimension $n$, where  $n, k\ge 2$. Then
$\lambda_{\max}(\mathcal{T})=\max\{\mathcal{T}{\bf x}^k:\|{\bf x}\|_k=1, {\bf x}\in \mathbb{R}^n\}$ and
$\lambda_{\min}(\mathcal{T})=\min\{\mathcal{T}{\bf x}^k:\|{\bf x}\|_k=1, {\bf x}\in \mathbb{R}^n\}$.
\end{Lemma}

It is not difficult to see that the previous lemma is not true if $k$ is odd.
Denote by $\mathbb{R}_+^n$
the set of all nonnegative vectors in $\mathbb{R}^n$.

\begin{Lemma}\label{RQ} \cite[Theorem~2]{Qi13}
Let $\mathcal{T}$ be a symmetric nonnegative tensor of order $k$ and dimension $n$, where $n,k\ge 2$. Then
$\lambda_{\max}(\mathcal{T})=\max\{\mathcal{T}{\bf x}^k:\|{\bf x}\|_k=1, {\bf x}\in \mathbb{R}_+^n\}$.
If $\lambda_{\max}(\mathcal{T})=\mathcal{T}{\bf x}^k$ for some ${\bf x}\in \mathbb{R}_+^n$ with
$\|{\bf x}\|_k=1$, then ${\bf x}$ is an $H$-eigenvector of $\mathcal{T}$ associated with $\lambda_{\max}(\mathcal{T})$.
\end{Lemma}

If  $\mathcal{T}$ be a symmetric  tensor of order $k$ and dimension $n$, where $k$ is even or if $\mathcal{T}$ be a symmetric, essentially nonnegative  tensor of order $k$ and dimension $n$, where $n,k\ge 2$, then
$\lambda_{\max}(\mathcal{T})=\max\{\mathcal{T}{\bf x}^k:\|{\bf x}\|_k=1, {\bf x}\in \mathbb{R}^n\}$.

Let $G$ be a $k$-uniform hypergraph with $V(G)=[n]$, and let ${\bf x}\in \mathbb{R}^n$. For $U\subseteq V(G)$, let ${\bf x}^U=\Pi_{w\in U}x_w$. Then
\[
\mathcal{A}(G){\bf x}^k=k\sum_{e\in E(G)} {\bf x}^e
\]
and for $u\in V(G)$,
\[
(\mathcal{A}(G){\bf x}^{k-1})_u
=\sum_{e\in E_u(G)} {\bf x}^{e\setminus\{u\}}.
\]
As $\mathcal{A}(G)$ is symmetric and nonnegative,  we have from Lemma  \ref{RQ} that
\[
\lambda(G)\le \mathcal{A}(G){\bf x}^k\le \rho(G).
\]

For a hypergraph $G$ with $V_1\subset V(G)$, $G-V_1$ denotes the hypergraph with vertex set $V(G)\setminus V_1$ and edge set $E(G)\setminus \{e:e\cap V_1=\emptyset \}$. If $V_1=\{v\}$, then we write $G-v$ for $G-\{v\}$.

For a hypergraph $G$ with $E_1\subseteq E(G)$, $G-E_1$ denotes the hypergraph with vertex set $V(G)$ and edge set $E(G)\setminus E_1$. If $E_1=\{e\}$, then we write $G-e$ for $G-\{e\}$.

For a symmetric tensor  $\mathcal{T}$ of order $k$ and dimension $n$ and $\emptyset \ne I\subset [n]$,
let $\mathcal{T}_I$ be the tensor of order $k$ and dimension $n$
such that
\[
(\mathcal{T}_I)_{i_1\ldots i_k}=\begin{cases}
t_{i_1\ldots i_k} & \hbox{if ~$\{i_1,\ldots, i_k\}\subseteq I$,} \\
0  & \hbox{otherwise.}
\end{cases}
\]

\begin{Lemma} \label{1OK} Let $\mathcal{T}$ be a zero diagonal symmetric tensor of order $k$ and dimension  $n$, where $n,k\ge 2$. If $k$ is even or $\mathcal{T}$ is nonnegative, then
$\lambda_{\max}(\mathcal{T})\geq 0$ and $\lambda_{\min}(\mathcal{T})\leq 0$.
\end{Lemma}

\begin{proof} Suppose first that $k$ is even. By Lemma \ref{sym}, for any ${\bf x}\in \mathbb{R}^n$ with
$\|{\bf x}\|_k=1$, we have
\[
\lambda_{\max}(\mathcal{T})\ge \mathcal{T}{\bf x}^k\ge \lambda_{\min}(\mathcal{T}).
\]
As $t_{1\dots 1}=0$, the coefficient of $x_1^k$ in $\mathcal{T}{\bf x}^k$ is $0$. Setting ${\bf y}=(1, 0, \dots, 0)^{\top}\in \mathbb{R}^n$, we have $\|{\bf y}\|_k=1$ and $\mathcal{T}{\bf y}^k=0$. So
$\lambda_{\max}(\mathcal{T})\ge 0\ge \lambda_{\min}(\mathcal{T})$.

Suppose next that $\mathcal{T}$ is nonnegative. By Lemma \ref{RQ}, $\lambda_{\max}(\mathcal{T})\ge 0$.
If $k$ is even, then by the above argument, $\lambda_{\min}(\mathcal{T})\le 0$. If $k$ is odd, then set ${\bf y}$ as above and we have $\mathcal{T}{\bf y}^{k-1}=0{\bf y}^{k-1}$ since $t_{1\dots 1}=0$, so $0$ is an H-eigenvalue of $\mathcal{T}$, implying that   $\lambda_{\min}(\mathcal{T})\le 0$.
\end{proof}

\begin{Lemma}\label{interlacing}
Let $\mathcal{T}$ be a zero diagonal symmetric tensor of order $k$ and dimension  $n$, where $n,k\ge 2$.
Let $\emptyset \ne I\subset [n]$. Suppose that $k$ is even  or $\mathcal{T}$ is nonnegative.
Then $\lambda_{\max}(\mathcal{T}[I])=\lambda_{\max}(\mathcal{T}_I)$.
\end{Lemma}

\begin{proof}
Let ${\bf x} \in \mathbb{R}^{|I|}$ be a unit eigenvector corresponding to $\lambda_{\max}(\mathcal{T}[I])$. Then
$\lambda_{\max}(\mathcal{T}[I])=\mathcal{T}[I]{\bf x}^k$.
Set $\widehat{{\bf x}}\in \mathbb{R}^n$  as a vector such that $\widehat{x}_i=x_i$ if $i\in I$,
and $\widehat{x}_i=0$ if $i\in [n]\setminus I$. Then it is easy to see that $\mathcal{T}[I]{\bf x}^k=\mathcal{T}_I\widehat{{\bf x}}^k$.
By Lemma \ref{sym} or \ref{RQ}, we have $\mathcal{T}_I\widehat{{\bf x}}^k \leq \lambda_{\max}(\mathcal{T}_I)$.
It follows that
\begin{equation}\label{a1}
 \lambda_{\max}(\mathcal{T}[I])=\mathcal{T}[I]{\bf x}^k=\mathcal{T}_I\widehat{{\bf x}}^k \leq \lambda_{\max}(\mathcal{T}_I).
\end{equation}
By Lemma \ref{1OK}, $\lambda_{\max}(\mathcal{T}[I])\geq 0$. If $\lambda_{\max}(\mathcal{T}_I)=0$, then we have from \eqref{a1}
that  $\lambda_{\max}(\mathcal{T}[I])=0=\lambda_{\max}(\mathcal{T}_I)$.
Suppose that $\lambda_{\max}(\mathcal{T}_I)>0$. From \eqref{a1}, we have $\lambda_{\max}(\mathcal{T}[I])\le \lambda_{\max}(\mathcal{T}_I)$.

Next, we  prove the converse inequality.  Let ${\bf y}\in \mathbb{R}^{n}$ be a unit eigenvector corresponding to $\lambda_{\max}(\mathcal{T}_I)$. Note that $\mathcal{T}$ is a zero diagonal tensor.
For $i\in [n]\setminus I$ and $i_2,\ldots, i_k\in [n]$,
as entries of $\mathcal{T}_I$, we have $t_{i\ldots i}=0$ and $t_{ii_2\ldots i_k}=0$. So
\[
\lambda_{\max}(\mathcal{T}_I)y_i^{k-1}=\sum_{i_2,\ldots, i_k\in [n]} t_{ii_2\dots i_k}y_{i_2}\cdots y_{i_k}=0
\]
for each $i\in [n]\setminus I$,  implying that $y_i=0$ for each $i\in [n]\setminus I$ as  $\lambda_{\max}(\mathcal{T}_I)>0$.
Let $\widehat{{\bf y}}\in \mathbb{R}^{|I|}$ such that $ \widehat{y}_i=y_i$ for $i\in I$.
 Note that $\widehat{{\bf y}}$ is unit.
For each $i\in I$, we have
\begin{align*}
(\mathcal{T}[I]\widehat{{\bf y}}^{k-1})_i&=\sum_{i_2,\ldots,i_k\in I}t_{i i_2 \ldots i_k}\widehat{y}_{i_2}\cdots \widehat{y}_{i_k}\\
&=\sum_{i_2,\ldots,i_k\in  I}t_{i i_2 \ldots i_k}y_{i_2}\cdots y_{i_k}\\
&=\sum_{i_2,\ldots,i_k\in [n]}t_{i i_2 \ldots i_k}y_{i_2}\cdots y_{i_k}\\
&=\lambda_{\max}(\mathcal{T}_I)y_i^{k-1}\\
&=\lambda_{\max}(\mathcal{T}_I)\widehat{y}_i^{k-1}.
\end{align*}
This means that  $\lambda_{\max}(\mathcal{T}_I)$ is an H-eigenvalue of $\mathcal{T}[I]$,
so $\lambda_{\max}(\mathcal{T}[I])\ge \lambda_{\max}(\mathcal{T}_I)$. Now it follows that
 $\lambda_{\max}(\mathcal{T}[I])=\lambda_{\max}(\mathcal{T}_I)$.
\end{proof}

\begin{Lemma}\label{S-interl}
Let $\mathcal{T}$ be a zero diagonal symmetric tensor of order even $k$ and dimension  $n$, where $n,k\ge 2$.
Let $\emptyset \ne I\subset [n]$.
Then  $\lambda_{\min}(\mathcal{T}[I])=\lambda_{\min}(\mathcal{T}_I)$.
\end{Lemma}

\begin{proof}
Let ${\bf x}\in \mathbb{R}^{|I|}$ be a unit eigenvector corresponding to $\lambda_{\min}(\mathcal{T}[I])$. Then
$\lambda_{\min}(\mathcal{T}[I])=\mathcal{T}[I]{\bf x}^k$.
Set $\widetilde{{\bf x}} \in \mathbb{R}^n$  as a vector such that $\widetilde{x}_i=x_i$ if $i\in I$,
and $\widetilde{x}_i=0$ if $i\in [n]\setminus I$. Then it is easy to see that $\mathcal{T}[I]{\bf x}^k=\mathcal{T}_I\widetilde{{\bf x}}^k$.
By Lemma \ref{sym}, we have $\lambda_{\min}(\mathcal{T}_I)\le \mathcal{T}_I\widetilde{{\bf x}}^k$.
Therefore
\[
 \lambda_{\min}(\mathcal{T}[I])=\mathcal{T}[I]{\bf x}^k=\mathcal{T}_I\widetilde{{\bf x}}^k \geq \lambda_{\min}(\mathcal{T}_I).
\]
By Lemma \ref{1OK}, $\lambda_{\min}(\mathcal{T}[I])\leq 0$. If $\lambda_{\min}(\mathcal{T}_I)=0$, then it follows from the above inequalities that  $\lambda_{\min}(\mathcal{T}_I)=0=\lambda_{\min}(\mathcal{T}[I])$.
Suppose that $\lambda_{\min}(\mathcal{T}_I)<0$.

Let ${\bf z}\in \mathbb{R}^n$ be a unit eigenvector corresponding to $\lambda_{\min}(\mathcal{T}_I)$.
As $t_{i\ldots i}=0$ and $t_{ii_2\ldots i_k}=0$  for each $i\in [n]\setminus I$ and $i_2,\ldots,i_k\in [n]$, we have $\lambda_{\min}(\mathcal{T}_I)z_i^{k-1}=0$ for each $i\in [n]\setminus I$. Then  $z_i=0$ for each $i\in [n]\setminus I$.
Let $\widehat{{\bf z}}$ be a vector in $\mathbb{R}^{|I|}$ such that $\widehat{z}_i=z_i$ for $i\in I$. Obviously, $\widehat{{\bf z}}$ is unit.
Then for each  $i\in I$,
\begin{align*}
(\mathcal{T}[I]\widehat{{\bf z}}^{k-1})_i
&=\sum_{i_2,\ldots,i_k\in I}t_{i i_2\ldots i_k}z_{i_2}\cdots z_{i_k}\\
&=\sum_{i_2,\ldots,i_k\in [n]}t_{i i_2\ldots i_k}z_{i_2}\cdots z_{i_k}\\
&=\lambda_{\min}(\mathcal{T}_I)\widehat{z}_i^{k-1}.
\end{align*}
This shows that $\lambda_{\min}(\mathcal{T}_I)$ is an H-eigenvalue of $\mathcal{T}[I]$.
Thus $\lambda_{\min}(\mathcal{T}[I])\leq \lambda_{\min}(\mathcal{T}_I)$. It follows that
$\lambda_{\min}(\mathcal{T}[I])=\lambda_{\min}(\mathcal{T}_I)$.
\end{proof}

We also need the well known combinatorial identity in the following lemma.

\begin{Lemma}\label{com}
Let $n$ and $k$ be positive integers with $1\leq k \leq n$.
Then
\[\sum_{i=k}^{n}{i \choose k}={n+1 \choose k+1}.\]
\end{Lemma}

\section{A formula for homogeneous polynomials of the form of inclusion-exclusion}

In this section we establish a formula for homogeneous polynomials of the type of Principle of Inclusion-Exclusion that will be used in the proofs.

\begin{Lemma}  \label{ling}
Let $\mathcal{T}$ be a zero diagonal nonzero symmetric tensor of order $k$ and dimension $n$, where $n,k\ge 2$. Let ${\bf x}$ be a unit  $n$-dimensional eigenvector and $\emptyset \ne I\subset [n]$.
For positive integers $s$ and $m$ with $s+m\leq k$,
\begin{equation}\label{meed}
\sum_{{i_1,\ldots,i_s\in [n]\setminus I\atop  i_{s+1},\ldots,i_{s+m}\in I}\atop i_{s+m+1},\ldots,i_k\in [n]}t_{i_1 \ldots i_k}x_{i_1}\cdots x_k=\sum_{\ell=0}^{m}(-1)^\ell{m\choose \ell}\sum_{i_1,\ldots,i_{s+\ell}\in [n]\setminus I\atop
i_{s+\ell+1},\ldots, i_k\in [n]}t_{i_1 \ldots i_k} x_{i_1}\cdots x_{i_k}.
\end{equation}
\end{Lemma}

\begin{proof}
We prove  identity \eqref{meed} by induction on $m$.
It is obvious that for any $s\geq 1$ with $s+1\leq k$,
\begin{align*}
\sum_{{i_1, \ldots,i_s\in [n]\setminus I\atop
i_{s+1}\in  I}\atop i_{s+2},\ldots,i_k\in [n]}
t_{i_1\ldots i_k}x_{i_1}\cdots x_k
&= \sum_{i_1,\ldots,i_s\in [n]\setminus I\atop i_{s+1},\ldots,i_k\in[n]}t_{i_1 \ldots  i_k} x_{i_1}\cdots x_{i_k}-\sum_{i_1,\ldots,i_{s+1}\in [n]\setminus I\atop i_{s+2},\ldots, i_k\in [n]} t_{i_1 \ldots i_k} x_{i_1}\cdots x_{i_k}\\
&=
\sum_{\ell=0}^{1}(-1)^\ell{1\choose \ell}\sum_{i_1,\ldots,i_{s+\ell}\in [n]\setminus I\atop i_{s+\ell+1},\ldots, i_k\in [n]}t_{i_1 \ldots i_k} x_{i_1}\cdots x_{i_k},
\end{align*}
proving \eqref{meed} when $m=1$.
Suppose that $1\le j<k-s$ and identity \eqref{meed} follows for $m=j$.
Suppose in the following that $m=j+1$. Then
\begin{equation}\label{zb}
\begin{aligned}
&\quad \sum_{{i_1,\ldots,i_s\in [n]\setminus I\atop
i_{s+1},\ldots,i_{s+m}\in I}\atop
i_{s+m+1},\ldots,i_k\in [n] }t_{i_1 \ldots i_k}x_{i_1}\cdots x_k\\
&=\sum_{{i_1,\ldots,i_s\in [n]\setminus I\atop
i_{s+1},\ldots,i_{s+j+1}\in I}\atop
 i_{s+j+2},\ldots,i_k\in [n] }t_{i_1 \ldots i_k}x_{i_1}\cdots x_k\\
&=\sum_{{i_1,\ldots,i_s\in [n]\setminus I\atop
 i_{s+1},\ldots,i_{s+j}\in I}\atop
  i_{s+j+1},\ldots,i_k\in [n] }t_{i_1 \ldots i_k}x_{i_1}\cdots x_k
-\sum_{{{i_1,\ldots,i_s\in [n]\setminus I\atop
i_{s+1},\ldots,i_{s+j}\in I}\atop
i_{s+j+1}\in [n]\setminus I}\atop
i_{s+j+2},\ldots,i_k\in [n]}t_{i_1\ldots i_k}x_{i_1}\cdots x_k\\
&=\sum_{{i_1,\ldots,i_s\in [n]\setminus I\atop
i_{s+1},\ldots,i_{s+j}\in I}\atop
i_{s+j+1},\ldots,i_k\in [n] }t_{i_1\ldots i_k}x_{i_1}\cdots x_k
-\sum_{{i_1,\ldots,i_{s+1}\in [n]\setminus I\atop
i_{s+2},\ldots,i_{s+j+1}\in I}\atop
 i_{s+j+2},\ldots,i_k\in [n]} t_{i_1\ldots i_k}x_{i_1}\cdots x_k,
\end{aligned}
\end{equation}
where the last equality in \eqref{zb} follows because $\mathcal{T}$ is symmetric. Now applying inductive hypothesis to the two summations of the last equation in \eqref{zb},  using combinatorial identity  ${j\choose \ell}={j+1\choose \ell}-{j\choose \ell-1}$ for $\ell\ge 1$, and by algebraic calculation to get
\begin{align*}
&\quad \sum_{{i_1,\ldots,i_s\in [n]\setminus I\atop
i_{s+1},\ldots,i_{s+m}\in I}\atop i_{s+m+1},\ldots,i_k\in [n]}t_{i_1 \ldots i_k}x_{i_1}\cdots x_k\\
&=\sum_{\ell=0}^{j}(-1)^\ell{j\choose \ell}\sum_{i_1,\ldots,i_{s+\ell}\in [n]\setminus I\atop i_{s+\ell+1},\ldots, i_k\in [n]}t_{i_1 \ldots i_k} x_{i_1}\cdots x_{i_k}-\sum_{\ell=0}^{j}(-1)^\ell{j\choose \ell}
\sum_{i_1,\ldots,i_{s+\ell+1}\in [n]\setminus I\atop
i_{s+\ell+2},\ldots, i_k\in [n]}t_{i_1 \ldots i_k} x_{i_1}\cdots x_{i_k}
\\
&= \sum_{i_1,\ldots,i_s\in [n]\setminus I\atop
i_{s+1},\ldots,i_k\in [n]}t_{i_1 \ldots  i_k} x_{i_1}\cdots x_{i_k}
+
\sum_{\ell=1}^{j}(-1)^\ell\left({j+1\choose \ell}-{j\choose \ell-1}\right)
\sum_{i_1,\ldots,i_{s+\ell}\in [n]\setminus I\atop
i_{s+\ell+1},\ldots, i_k\in [n]}t_{i_1 \ldots i_k} x_{i_1}\cdots x_{i_k}\\
&-\sum_{i_1,\ldots,i_{s+1}\in [n]\setminus I\atop
i_{s+2},\ldots,i_k\in[n]}t_{i_1 \ldots  i_k} x_{i_1}\cdots x_{i_k}-\sum_{\ell=2}^{j+1}(-1)^{\ell-1}{j\choose \ell-1}
\sum_{i_1,\ldots,i_{s+\ell}\in [n]\setminus I\atop
i_{s+\ell+1},\ldots, i_k\in [n]}t_{i_1 \ldots i_k} x_{i_1}\cdots x_{i_k}  \\
&= \sum_{i_1,\ldots,i_s\in [n]\setminus I\atop
i_{s+1},\ldots,i_k\in[n]}t_{i_1 \ldots  i_k} x_{i_1}\cdots x_{i_k}
+
\sum_{\ell=1}^{j}(-1)^\ell{j+1\choose \ell}
\sum_{i_1,\ldots,i_{s+\ell}\in [n]\setminus I\atop
i_{s+\ell+1},\ldots, i_k\in [n]}t_{i_1 \ldots i_k} x_{i_1}\cdots x_{i_k}\\
&+(-1)^{j+1}\sum_{i_1,\ldots,i_{s+j+1}\in [n]\setminus I\atop
i_{s+j+2},\ldots, i_k\in [n]}t_{i_1 \ldots i_k} x_{i_1}\cdots x_{i_k}.
\end{align*}
Now we have
\[\sum_{{i_1,\ldots,i_s\in [n]\setminus I\atop
i_{s+1},\ldots,i_{s+m}\in I}\atop
i_{s+m+1},\ldots,i_k\in [n]}t_{i_1 \ldots i_k}x_{i_1}\cdots x_k
=
\sum_{\ell=0}^{j+1}(-1)^\ell{j+1\choose \ell}
\sum_{i_1,\ldots,i_{s+\ell}\in [n]\setminus I\atop i_{s+\ell+1},\ldots, i_k\in [n]}t_{i_1 \ldots i_k} x_{i_1}\cdots x_{i_k}.
\]
This proves \eqref{meed} when $m=j+1$. By induction, \eqref{meed} follows.
\end{proof}

\section{Largest H-eigenvalues}

We are now ready to give our results on the largest H-eigenvalues of symmetric tensors and uniform hypergraphs. First, we give the interlacing inequalities for the largest H-eigenvalues.

\begin{Theorem}\label{bound-1}
Let $\mathcal{T}$ be a zero diagonal nonzero symmetric tensor of order $k$ and dimension $n$, where $n,k\ge 2$.  Suppose that $k$ is even or $\mathcal{T}$ is nonnegative. Let ${\bf x}$ be a unit  eigenvector corresponding to $\lambda_{\max}(\mathcal{T})$.
Let $\emptyset\ne I\subset [n]$. Then
\begin{align*}
\lambda_{\max}(\mathcal{T})\left(1-k\sum_{i\in [n]\setminus I}x_i^k\right)-\sum_{j=1}^{k-1}(-1)^j{k \choose j+1}\sum_{i_1,\ldots,i_{j+1}\in [n]\setminus I  \atop i_{j+2},\ldots,i_k\in[n]}t_{i_1 \ldots  i_k}x_{i_1}\cdots x_{i_k} &\leq \lambda_{\max}(\mathcal{T}[I])\\
&\leq \lambda_{\max}(\mathcal{T}).
\end{align*}
\end{Theorem}

\begin{proof}
As $\mathcal{T}$ is symmetric, we have by Lemma \ref{ling} that,
for $1\leq j\leq k-1$, one has
\begin{align*}
\sum_{{i_1,\ldots,i_j\in I\atop
i_{j+1}\in [n]\setminus I}\atop
i_{j+2},\ldots,i_k\in [n]}t_{i_1 \ldots i_k} x_{i_1}\cdots x_{i_k}
&=\sum_{{i_1\in [n]\setminus I\atop
i_2,\ldots,i_{j+1}\in I}\atop
i_{j+2},\ldots,i_k\in [n]}t_{i_1 \ldots i_k} x_{i_1}\cdots x_{i_k}\\
&= \sum_{\ell=0}^{j}(-1)^\ell{j\choose \ell}\sum_{i_1,\ldots,i_{\ell+1}\in [n]\setminus I \atop i_{\ell+2},\ldots, i_k\in [n]}t_{i_1 \ldots i_k} x_{i_1}\cdots x_{i_k}.
\end{align*}
For any $i_1\in [n]\setminus I$, as $\lambda_{\max}(\mathcal{T})$ is an H-eigenvalue of $\mathcal{T}$
associated to eigenvector ${\bf x}$, we have
\[
\sum_{i_2,\ldots, i_k\in [n]} t_{i_1 \ldots i_k} x_{i_1}\cdots x_{i_k}=\lambda_{\max}(\mathcal{T})x_{i_1}^{k}
\]
and hence
\[
\sum_{i_1\in [n]\setminus I \atop i_2,\ldots, i_k\in [n]} t_{i_1 \ldots i_k} x_{i_1}\cdots x_{i_k}=
\lambda_{\max}(\mathcal{T})\sum_{i_1\in [n]\setminus I} x_{i_1}^{k}.
\]
Therefore,
\begin{align*}
&\quad (\mathcal{T}-\mathcal{T}_I){\bf x}^k\\
&=\sum_{i_1,\ldots, i_k\in[n]} t_{i_1 \ldots  i_k} x_{i_1}\cdots x_{i_k}-\sum_{i_1,\ldots, i_k\in I} t_{i_1 \ldots  i_k} x_{i_1}\cdots x_{i_k}\\
&=\sum_{i_1\in [n]\setminus I\atop i_2,\ldots, i_k\in [n]} t_{i_1\ldots i_k} x_{i_1}\cdots x_{i_k}
+\sum_{j=1}^{k-1}\sum_{{i_1,\ldots,i_j\in I\atop
i_{j+1}\in [n]\setminus I}\atop
i_{j+2},\ldots,i_k\in [n]}t_{i_1 \ldots i_k} x_{i_1}\cdots x_{i_k}\\
&=\sum_{i_1\in [n]\setminus I \atop i_2,\ldots, i_k\in [n]} t_{i_1 \ldots  i_k} x_{i_1}\cdots x_{i_k}+\\
&+\sum_{j=1}^{k-1}\left(\sum_{i_1\in [n]\setminus I \atop i_2,\ldots,i_k\in[n]}t_{i_1\ldots i_k} x_{i_1}\cdots x_{i_k}+\sum_{\ell=1}^{j}(-1)^\ell{j\choose \ell}\sum_{i_1,\ldots,i_{\ell+1}\in [n]\setminus I \atop i_{\ell+2},\ldots, i_k\in [n]}t_{i_1 \ldots i_k} x_{i_1}\cdots x_{i_k}\right)\\
&=k\sum_{i_1\in [n]\setminus I \atop i_2,\ldots, i_k\in [n]} t_{i_1 \ldots i_k} x_{i_1}\cdots x_{i_k}+\sum_{j=1}^{k-1}\sum_{\ell=1}^{j}(-1)^\ell{j\choose \ell}\sum_{i_1,\ldots,i_{\ell+1}\in [n]\setminus I \atop i_{\ell+2},\ldots, i_k\in [n]}t_{i_1 \ldots i_k} x_{i_1}\cdots x_{i_k}.
\end{align*}
By interchanging the order of summation in
\[
\sum_{j=1}^{k-1}\sum_{\ell=1}^{j}(-1)^\ell{j\choose \ell}\sum_{i_1,\ldots,i_{\ell+1}\in [n]\setminus I \atop i_{\ell+2},\ldots, i_k\in [n]}t_{i_1 \ldots i_k} x_{i_1}\cdots x_{i_k},
\]
one has
\[
(\mathcal{T}-\mathcal{T}_I){\bf x}^k=k\lambda_{\max}(\mathcal{T})\sum_{i\in [n]\setminus I} x_i^k+\sum_{j=1}^{k-1}(-1)^j\sum_{\ell=j}^{k-1}{\ell\choose j}\sum_{i_1,\ldots,i_{j+1}\in [n]\setminus I \atop i_{j+2},\ldots, i_k\in [n]}t_{i_1 \ldots  i_k} x_{i_1}\cdots x_{i_k}.
\]
Now, by Lemma~\ref{com}, one has
\begin{align}
(\mathcal{T}-\mathcal{T}_I){\bf x}^k
&=k\lambda_{\max}(\mathcal{T})\sum_{i\in [n]\setminus I} x_i^k+\sum_{j=1}^{k-1}(-1)^j{k\choose j+1}\sum_{i_1,\ldots,i_{j+1}\in [n]\setminus I \atop i_{j+2},\ldots, i_k\in [n]}t_{i_1 \ldots  i_k} x_{i_1}\cdots x_{i_k}.\label{Eq-22-4-11}
\end{align}
Thus, one has by Lemma~\ref{sym} or~\ref{RQ} that
\begin{align*}
\lambda_{\max}(\mathcal{T}_I)
& \geq \mathcal{T}_I{\bf x}^k
= \mathcal{T}{\bf x}^k- (\mathcal{T}-\mathcal{T}_I){\bf x}^k\\
&=\lambda_{\max}(\mathcal{T})-\left(k\lambda_{\max}(\mathcal{T})\sum_{i\in [n]\setminus I} x_i^k+\sum_{j=1}^{k-1}(-1)^j{k\choose j+1}\sum_{i_1,\ldots,i_{j+1}\in [n]\setminus I \atop i_{j+2},\ldots, i_k\in [n]}t_{i_1 \ldots i_k} x_{i_1}\cdots x_{i_k}\right)\\
&=\lambda_{\max}(\mathcal{T})\left(1-k\sum_{i\in [n]\setminus I} x_i^k\right)-\sum_{j=1}^{k-1}(-1)^j{k\choose j+1}\sum_{i_1,\ldots,i_{j+1}\in [n]\setminus I \atop i_{j+2},\ldots, i_k\in [n]}t_{i_1 \ldots i_k} x_{i_1}\cdots x_{i_k}.
\end{align*}
Now,  the first inequality follows from the above equation and Lemma~\ref{interlacing}.

Let ${\bf y}\in \mathbb{R}^{|I|}$ be a unit eigenvector corresponding to $\lambda_{\max}(\mathcal{T}[I])$.
Construct a unit vector ${\bf z}\in \mathbb{R}^n$ such that $z_i=y_i$ if $i\in I$
and $z_i=0$ if $i\in [n]\setminus I$. Note that $T_I {\bf z}^k= T[I]{\bf y}^k$.
As $\mathcal{T}$ is zero diagonal, and
$(\mathcal{T}-\mathcal{T}_I)_{i_1\ldots i_k}=t_{i_1\ldots i_k}$ if $\{i_1,\ldots ,i_k\}\cap ([n]\setminus I)\neq \emptyset$ and $0$ otherwise,
we have
$(\mathcal{T}-\mathcal{T}_I){\bf z}^k=0$.
Thus by Lemma~\ref{sym} or~\ref{RQ},
\begin{align*}
\lambda_{\max}(\mathcal{T})\geq \mathcal{T}{\bf z}^k
=\mathcal{T}_I{\bf z}^k+ (\mathcal{T}-\mathcal{T}_I){\bf z}^k
= T[I]{\bf y}^k+0
= \lambda_{\max}(\mathcal{T}[I]). 
\end{align*}
This proves the second inequality.
\end{proof}

\begin{Example}
Let $\mathcal{T}$ be a tensor of order $4$ and dimension $3$, where
$t_{1122}=t_{1212}=t_{1221}=t_{2121}=t_{2211}=t_{2112}=-1$,
$t_{1222}=t_{2122}=t_{2212}=t_{2221}=\frac{1}{2}$,
$t_{3222}=t_{2322}=t_{2232}=t_{2223}=1$,
and otherwise,
$t_{ijst}=0$.
Let $I=\{2,3\}$. By MATLAB, we have $\lambda_{\max}(\mathcal{T})=2.4043$ with eigenvector ${\bf x}_0=(0.1632,1,0.7465)^\top$ and $\lambda_{\max}(\mathcal{T}[I])=2.2795$.
Note that ${\bf x}_0$ is not unit.
Let ${\bf x}= \frac{{\bf x}_0}{\|{\bf x}_0\|_4}$.
The lower bound for $\lambda_{\max}(\mathcal{T}[I])$ in Theorem~\ref{bound-1} is equal to
\begin{align*}
&\quad \lambda_{\max}(\mathcal{T})\left(1-k\sum_{i\in [n]\setminus I}x_i^k\right)-\sum_{j=1}^{k-1}(-1)^j{k \choose j+1}\sum_{i_1,\ldots,i_{j+1}\in [n]\setminus I \atop i_{j+2},\ldots,i_k\in[n]}t_{i_1 \ldots  i_k}x_{i_1}\cdots x_{i_k}\\
&=\lambda_{\max}(\mathcal{T})\left(1-4x_1^4\right)-(-1)^1{4 \choose 2}t_{1122}x_1^2x_2^2\\
&=2.4043 \times\left(1- 4\times \frac{0.1632^4}{0.1632^4+1+0.7465^4}\right)+{4\choose 2}\times \frac{-1\times 0.1632^2 \times 1^2}{0.1632^4+1+0.7465^4}\\
&=2.2772.
\end{align*}
\end{Example}

\begin{Example}
Let $\mathcal{T}$ be a tensor of order $3$ and dimension $5$, where
$t_{112}=t_{121}=t_{211}=\frac{1}{3}$, $t_{122}=t_{221}=t_{212}=\frac{1}{12}$,
$t_{113}=t_{131}=t_{311}=\frac{1}{6}$,
$t_{223}=t_{232}=t_{322}=\frac{1}{12}$, $t_{233}=t_{323}=t_{332}=\frac{1}{18}$,
$t_{123}=t_{132}=t_{213}=t_{231}=t_{321}=t_{312}=-\frac{1}{12}$
$t_{445}=t_{454}=t_{544}=\frac{1}{6}$,
and otherwise,
$t_{ijst}=0$.
Let $I=\{1,2,4\}$. By MATLAB, we have $\lambda_{\max}(\mathcal{T})=0.6894$ with eigenvector ${\bf x}_0=(1,0.8241,0.4256,0,0)^\top$ and $\lambda_{\max}(\mathcal{T}[I])=0.6387$.
Let ${\bf x}=\frac{{\bf x}_0}{\|{\bf x}_0\|_3}$
The lower bound for $\lambda_{\max}(\mathcal{T}[I])$ in Theorem~\ref{bound-1} is equal to
\begin{align*}
&\quad \lambda_{\max}(\mathcal{T})\left(1-3(x_3^3+x_5^3)\right)-(-1)^1{3 \choose 2}t_{332}x_3^2x_2\\
&=0.6894 \times \left(1- \frac{3\times 0.4256^3}{1+0.8241^3+0.4256^3}\right)+{3\choose 2}\times\frac{\frac{1}{18}\times 0.4256^2 \times 0.8241}{1^3+0.8241^3+0.4256^3}\\
&=0.6072.
\end{align*}

\end{Example}

We remark that the second inequality has been known when $\mathcal{T}$ is nonnegative \cite{HHQ,KS}.

\begin{Theorem}\label{bound-10}
Let $\mathcal{T}$ be a symmetric, nonnegative zero diagonal tensor of order $k$
and dimension $n$, where $k,n \geq 2$ and ${\bf x}$ be a unit nonnegative eigenvector corresponding to $\lambda_{\max}(\mathcal{T})$.
Let $\emptyset\ne I\subset [n]$. If  $\sum_{i\in I} x_i^k\ne 0$, then
\begin{align*}
\frac{ \lambda_{\max}(\mathcal{T})\left(1-k\sum_{i\in [n]\setminus I} x_i^k\right)-\sum_{j=1}^{k-1}(-1)^j{k\choose j+1}\sum_{i_1,\ldots,i_{j+1}\in [n]\setminus I \atop i_{j+2},\ldots, i_k\in [n]}t_{i_1 \ldots i_k} x_{i_1}\cdots x_{i_k}}{\sum_{i\in  I} x_i^k} & \leq  \rho(\mathcal{T}[I])\\
&  \leq  \rho(\mathcal{T}).
\end{align*}
\end{Theorem}

\begin{proof} As $\mathcal{T}$ is nonnegative, we have by Proposition \ref{PF} that $\lambda_{\max}(\mathcal{T})=\rho(\mathcal{T})$ and $\lambda_{\max}(\mathcal{T}[I])=\rho(\mathcal{T}[I])$.

It is evident that the upper bound of $\rho(\mathcal{T}[I])$ follows from  Theorem~\ref{bound-1}, see also \cite{KS}.

By the proof of Theorem~\ref{bound-1},
we have
\begin{align*}
\mathcal{T}_I{\bf x}^k
&=\rho(\mathcal{T})\left(1-k\sum_{i\in [n]\setminus I} x_i^k\right)-\sum_{j=1}^{k-1}(-1)^j{k\choose j+1}\sum_{i_1,\ldots,i_{j+1}\in [n]\setminus I \atop i_{j+2},\ldots, i_k\in [n]}t_{i_1 \ldots i_k} x_{i_1}\cdots x_{i_k}.
\end{align*}
Applying  Lemma~\ref{RQ} by setting ${\bf y}\in \mathbb{R}^{|I|}$ with $y_i=x_i$ for $i\in I$, we have
\begin{align*}
\rho(\mathcal{T}[I]) & \geq \frac{ \mathcal{T}[I]{\bf y}^k}{\|{\bf y}\|_k^k}
=\frac{\mathcal{T}_I{\bf x}^k}{\sum_{i\in  I} x_i^k}\\
&=\frac{ \rho(\mathcal{T})\left(1-k\sum_{i\in [n]\setminus I} x_i^k\right)-\sum_{j=1}^{k-1}(-1)^j{k\choose j+1}\sum_{i_1,\ldots,i_{j+1}\in [n]\setminus I \atop i_{j+2},\ldots, i_k\in [n]}t_{i_1 \ldots i_k} x_{i_1}\cdots x_{i_k}}{\sum_{i\in I} x_i^k},
\end{align*}
proving the lower bound part.
\end{proof}

\begin{Example}
Let $\mathcal{T}$ be a tensor of order $3$ and dimension $3$, where
$t_{112}=t_{121}=t_{211}=\frac{1}{3}$, $t_{122}=t_{221}=t_{212}=\frac{1}{12}$,
$t_{113}=t_{131}=t_{311}=\frac{1}{6}$,
$t_{223}=t_{232}=t_{322}=\frac{1}{12}$, $t_{233}=t_{323}=t_{332}=\frac{1}{18}$
and otherwise,
$t_{ijst}=0$.
Let $I=\{1,2\}$. By MATLAB, we have $\lambda_{\max}(\mathcal{T})=0.8143$ with eigenvector ${\bf x}_0=(1,0.84,0.5866)^\top$ and $\lambda_{\max}(\mathcal{T}[I])=0.6387$. Let ${\bf x}=\frac{{\bf x}_0}{\|{\bf x}_0\|}$.
The lower bound for $\lambda_{\max}(\mathcal{T}[I])$ in Theorem~\ref{bound-10} is equal to
\begin{align*}
&\quad \frac{ \lambda_{\max}(\mathcal{T})\left(1-3 x_3^3\right)-(-1)^1{3\choose 2}t_{331} x_3^2x_1}{ x_1^3+x_2^3}\\
&= \frac{0.8143\times \left(1-3\times \frac{0.5866^3}{1+0.84^3+0.5866^3}\right)+{3\choose 2}\times \frac{1}{18}\times \frac{0.5866^2\times 0.84}{1+0.84^3+0.5866^3}}{\frac{1+0.84^3}{1+0.84^3+0.5866^3}
}\\
&=0.6381.
\end{align*}
\end{Example}

When the tensor $\mathcal{T}$ is weakly irreducible and nonnegative, by Proposition \ref{PF}, there is a unique unit positive eigenvector corresponding to $\lambda_{\max}(\mathcal{T})=\rho(\mathcal{T})$.
By the proof of Theorem \ref{bound-1}, $\rho(\mathcal{T}[I])<\rho(\mathcal{T})$. This together with Theorem \ref{bound-10} implies the following corollary.

\begin{Corollary}\label{bound-1+}
Let $\mathcal{T}$ be a weakly irreducible, symmetric, nonnegative zero diagonal tensor
of order $k$ and dimension $n$, where $k,n\geq 2$ and ${\bf x}$ be a unit positive eigenvector corresponding to $\rho(\mathcal{T})$.
For $\emptyset\ne I\subset [n]$,
\begin{align*}
\frac{ \rho(\mathcal{T})\left(1-k\sum_{i\in [n]\setminus I} x_i^k\right)-\sum_{j=1}^{k-1}(-1)^j{k\choose j+1}\sum_{i_1,\ldots,i_{j+1}\in [n]\setminus I \atop i_{j+2},\ldots, i_k\in [n]}t_{i_1 \ldots i_k} x_{i_1}\cdots x_{i_k}}{\sum_{i\in  I} x_i^k} & \leq \rho(\mathcal{T}[I])\\
 & <\rho(\mathcal{T}).
\end{align*}
\end{Corollary}

For a tensor $\mathcal{T}$
of order $k$ and dimension $n$ with $k,n\geq 2$, we denote by $R_i(\mathcal{T})=\sum_{i_2,\dots, i_k\in [n]}t_{ii_2\dots i_k}$ for $i\in[n]$, which is called the $i$th row sum of  $\mathcal{T}$.

\begin{Example}
Let $\mathcal{T}$ be a weakly irreducible, symmetric, nonnegative zero diagonal tensor
of order $k$ and dimension $n$, where $k,n\geq 2$. Suppose that $R_i(\mathcal{T})=\dots =R_n(\mathcal{T})=r$. By Proposition \ref{PF}, ${\bf x}=n^{-\frac{1}{k}}(1,\dots, 1)^{\top}$ is the unit positive eigenvector corresponding to $\rho(\mathcal{T})$.
Let  $\emptyset\ne I\subset [n]$. Then
\[
\sum_{i\in [n]\setminus I} x_i^k=\frac{n-|I|}{n}
\]
and
\[
\sum_{j=1}^{k-1}(-1)^j{k\choose j+1}\sum_{i_1,\ldots,i_{j+1}\in [n]\setminus I \atop i_{j+2},\ldots, i_k\in [n]}t_{i_1 \ldots i_k} x_{i_1}\cdots x_{i_k}=\frac{1}{n}\sum_{j=1}^{k-1}(-1)^j{k\choose j+1}\sum_{i_1,\ldots,i_{j+1}\in [n]\setminus I \atop i_{j+2},\ldots, i_k\in [n]}t_{i_1 \ldots i_k}.
\]
By Corollary \ref{bound-1+},
\begin{align}
\rho(\mathcal{T}[I])\ge \frac{ \rho(\mathcal{T})\left(n-k(n-|I|)\right)-\sum_{j=1}^{k-1}(-1)^j{k\choose j+1}\sum_{i_1,\ldots,i_{j+1}\in [n]\setminus I \atop i_{j+2},\ldots, i_k\in [n]}t_{i_1 \ldots i_k} }{|I|}.\label{eq22-6-14}
\end{align}

Suppose further that $\mathcal{T}$ is a tensor of order $3$ and dimension $3$, where
$t_{121}=t_{112}=t_{211}=t_{233}=t_{323}=t_{332}=\frac{1}{3}$,
$t_{131}=t_{113}=t_{311}=t_{232}=t_{223}=t_{322}=\frac{1}{6}$,
and otherwise,
$t_{ij\ell}=0$.
It is easily checked that $\mathcal{T}$ is weakly irreducible with
$R_1(\mathcal{T})=R_2(\mathcal{T}) =R_3(\mathcal{T})=1$.
Then $\rho(\mathcal{T})=1$ with eigenvector $(\frac{1}{\sqrt[3]{3}},\frac{1}{\sqrt[3]{3}},\frac{1}{\sqrt[3]{3}})^\top$.
Let $I=\{1,2\}$. Let ${\bf y}=(y_1,y_2)$ be a unit  positive eigenvector corresponding to $\rho(\mathcal{T}[I])$.
Then 
$\rho(\mathcal{T}[I])y_1^2=\frac{2}{3}y_1y_2$  and
$\rho(\mathcal{T}[I])y_2^2=\frac{1}{3}y_1^2$,
from which it follows that $\rho(\mathcal{T}[I])^3=\frac{4}{27}$.
So $\rho(\mathcal{T}[I])=\frac{\sqrt[3]{4}}{3}\approx 0.5291$. This is compared to
 the lower bound for $\rho(\mathcal{T}[I])$ given by  \eqref{eq22-6-14}, which is equal to $0.5$.
\end{Example}

Note that the spectral radius of an uniform hypergraph is at least $0$.
For hypergraphs, we have the following result, which generalizes/improves the result of
\cite[Theorem 1]{LWM}.

\begin{Theorem}\label{up-1}
Let $G$ be a $k$-uniform hypergraph with vertex set $[n]$, and
${\bf x}$ be a unit nonnegative eigenvector corresponding to $\rho(G)$, where $n\geq k\geq 2$.
Let  $\emptyset\ne I\subset [n]$. Then
\[
\rho(G)\left(1-k\sum_{i\in I}x_i^k\right)+k\sum_{j=2}^{k}\sum_{e:|e\cap I|=j}(j-1){\bf x}^e\leq \rho(G-I) \leq  \rho(G).
\]
Moreover, if $\sum_{i\in I}x_i^k\ne 1$ (for example, if $G$ is connected), then
\[
\rho(G)\frac{\left(1-k\sum_{i\in I}x_i^k\right)+k\sum_{j=2}^{k}\sum_{e:|e\cap I|=j}(j-1){\bf x}^e}{1-\sum_{i\in I}x_i^k}\leq \rho(G-I) \leq  \rho(G).
\]
\end{Theorem}

\begin{proof}
Since $\mathcal{A}(G)$ and $\mathcal{A}(G-I)$ is nonnegative, we have $\rho(G)=\lambda_{\max}(\mathcal{A}(G))$ and $\rho(G-I)=\lambda_{\max}(\mathcal{A}(G-I))$.
By Theorem~\ref{bound-1},
we have
\begin{equation}\label{eq22-4-22}
\rho(G)\left(1-k\sum_{i\in I}x_i^k\right)+\sum_{j=1}^{k-1}(-1)^j{k \choose j+1}\sum_{i_1,\ldots,i_{j+1}\in I \atop i_{j+2},\ldots,i_k\in[n]}a_{i_1 \ldots  i_k}x_{i_1}\cdots x_{i_k}\leq \rho(G-I) \leq  \rho(G).
\end{equation}
Let $e=\{i_1,\ldots,i_k\}\in E(G)$ with $e\cap I=\{i_1,\ldots ,i_\ell\}$, where $0\leq \ell\leq k$.
The coefficient of
${\bf x}^e$ in the lower bound given by \eqref{eq22-4-22} is $0$ if $\ell=0,1$.
If $\ell\geq 2$, it is
\begin{align*}
&-\sum_{j=1}^{\ell-1}(-1)^j{k\choose j+1}\left({\ell \choose j+1}(j+1)!(k-j-1)!\right)\frac{1}{(k-1)!}\\
&=-k\sum_{j=1}^{\ell-1}(-1)^j{\ell\choose j+1}\\
&=k\sum_{j=2}^{\ell}(-1)^j{\ell\choose j}\\
&=k(\ell-1).
\end{align*}
Thus
\begin{equation} \label{Eq-22-4-13}
-\sum_{j=1}^{k-1}(-1)^j{k\choose j+1}\sum_{i_1,\ldots,i_{j+1}\in I \atop i_{j+2},\ldots, i_k\in [n]}a_{i_1 \ldots i_k} x_{i_1}\cdots x_{i_k}
=k\sum_{\ell=2}^{k}\sum_{e:|e\cap I|=\ell}(\ell-1){\bf x}^e.
\end{equation}
Now, by \eqref{eq22-4-22} and \eqref{Eq-22-4-13},  the first part follows.

Suppose that $\sum_{i\in I}x_i^k\ne 1$.
By Theorem~\ref{bound-1+} and the above argument, the second part follows.
\end{proof}

\begin{Example}
Let $G$ be a connected $k$-uniform regular  hypergraph with vertex set $[n]$. Then
${\bf x}=n^{-\frac{1}{k}}(1,\dots,1)^{\top}$ is the  unit nonnegative eigenvector corresponding to $\rho(G)$.
Let  $\emptyset\ne I\subset [n]$. Then
\[
\rho(G-I)\ge \rho(G)\frac{\left(n-k|I|\right)+k\sum_{j=2}^{k}\sum_{e:|e\cap I|=j}(j-1)}{n-|I|}
\]
\end{Example}

\begin{Corollary}\label{up-3}
Let $G$ be a connected $k$-uniform hypergraph with vertex set $[n]$, and
${\bf x}$ be a unit positive eigenvector corresponding to $\rho(G)$, where $n\geq k\geq 2$.
Let  $\emptyset \ne I\subseteq[n]$. Then
\[
\sum_{i\in I}x_i^k\leq \frac{1}{k}+ \frac{1}{\rho(G)}\sum_{j=2}^{k}\sum_{e:|e\cap I|=j}(j-1){\bf x}^e.
\]
In particular, for any $v\in V(G)$,
\begin{align}
x_v \leq \sqrt[k]{\frac{1}{k}}.\label{Eq-22-5-18}
\end{align}
\end{Corollary}
\begin{proof}
Let $\mathcal{T}=\mathcal{A}(G)$.
By~\eqref{Eq-22-4-11} and~\eqref{Eq-22-4-13},
we have \[(\mathcal{T}-\mathcal{T}_{[n]\setminus I}){\bf x}^k=k\rho(G)\sum_{i\in I}x_i^k-k\sum_{j=2}^{k}\sum_{e:|e\cap I|=j}(j-1){\bf x}^e.\]
Then
\begin{align*}
0 &\leq \mathcal{T}_{[n]\setminus I}{\bf x}^k\\
&=(\mathcal{T}-(\mathcal{T}-\mathcal{T}_{[n]\setminus I})){\bf x}^k\\
&=\rho(G)-\left(k\rho(G)\sum_{i\in I}x_i^k-k\sum_{j=2}^{k}\sum_{e:|e\cap I|=j}(j-1){\bf x}^e\right)\\
&=\left(1-k\sum_{i\in I}x_i^k\right)\rho(G)+k\sum_{j=2}^{k}\sum_{e:|e\cap I|=j}(j-1){\bf x}^e.
\end{align*}
The result follows.
\end{proof}

We remak that \eqref{Eq-22-5-18} has been observed in~\cite[Proposition~7.21]{Niki}.

\begin{Theorem}\label{up-2}
Let $G$ be a $k$-uniform hypergraph with vertex set $[n]$, and
${\bf x}$ be a unit nonnegative eigenvector corresponding to $\rho(G)$, where $n\geq k\geq 2$.
Let  $v\in V(G)$ and $\mathbf{y}$ be the restriction of $\mathbf{x}$ on $[n]\setminus\{v\}$. Then
\[
\left(1-kx_v^k\right)\rho(G)\leq \rho(G-v) \leq  \rho(G).
\]
Moreover, if $x_v^k\ne 1$, then
\begin{equation}\label{HY}
 \rho(G-v) \ge \frac{1-kx_v^k}{1-x_v^k}\rho(G),
\end{equation}
if  $G$ is connected, then equality holds in \eqref{HY}
if and only if
$\mathbf{y}$ is an eigenvector of $G-v$ associated with $\rho(G-v)$.
\end{Theorem}

\begin{proof}
Taking $I=\{v\}\subseteq V(G)$ in Theorem~\ref{up-1}, we immediately have the inequalities.

Suppose that $G$ is connected. By Proposition \ref{PF}, $\mathbf{x}$ is positive, and then $1-x_v^k>0$.
Note that
\begin{align*}
\rho(G)=k\sum_{e\in E(G)\setminus E_v(G)}\mathbf{x}^e+k\sum_{e\in E_v (G)}\mathbf{x}^e=\mathcal{A}(G-v)\mathbf{y}^{k}+k\rho(G)x_v^k.
\end{align*}
From this, it is easy to see that
\[(1-kx_v^k)\rho(G)=\mathcal{A}(G-v)\mathbf{y}^{k}\Leftrightarrow \rho(G-v)
=\frac{\mathcal{A}(G-v)\mathbf{y}^k}{1-x_v^k}.
\]
So, by Lemma~\ref{RQ}, equality holds in \eqref{HY}
if and only if
$\mathbf{y}$ is an eigenvector of $G-v$ associated with $\rho(G-v)$.
\end{proof}


If $G$ is a connected graph with at least two vertices, then for any $v\in V(G)$,
$\rho(G-v)\ge \frac{1-2x_v^2}{1-x_v^2} \rho(G)$,
which is known \cite{Se}.
If $v$ is the vertex  such that $x_v$ is minimum,
then the inequality $\rho(G-v)\ge \frac{1-kx_v^k}{1-x_v^k}$ was known  \cite{KLS} (for $k=2$ \cite{Nik}).

\begin{Theorem}\label{ggv}
Let $G$ be a connected $k$-uniform linear hypergraph with vertex set $[n]$, where $n\ge k\ge 2$.
For any $v\in [n]$,
\[
\rho(G-v)\ge \rho(G)-\sqrt[k-1]{\frac{d_G(v)}{\rho(G)}}
\]
with equality if and only if $v$ is adjacent to any other vertex, and $G-v$ is regular.
\end{Theorem}

\begin{proof} Let $\mathbf{x}$ be the unit positive eigenvector corresponding to $\rho(G)$ and $\mathbf{y}$ be the restriction of $\mathbf{x}$ on $V(G)\setminus\{v\}$.
By Theorem \ref{up-2},
\[
 \rho(G-v) \ge \frac{1-kx_v^k}{1-x_v^k}\rho(G)
\]
with equality if and only if
$\mathbf{y}$ is an eigenvector of $G-v$ associated with $\rho(G-v)$.
Let $\rho=\rho(G)$ and $d=d_G(v)$. As $G$ is linear, we have by~\cite[Theorem~3.1]{LKY} and its proof that
\[
x_v\le\frac{1}{\sqrt[k]{1+(k-1)\left(\frac{\rho^k}{d}\right)^{\frac{1}{k-1}}}}
\]
with equality  only if  the entries of $\mathbf{y}$ are all equal.
So
\[
\rho(G-v)\ge \frac{\rho^\frac{k}{k-1}-d^{\frac{1}{k-1}}}{\rho^{\frac{1}{k-1}}},
\]
from which we have
\[
\rho(G)-\rho(G-v)\le \rho-\frac{\rho^\frac{k}{k-1}-d^{\frac{1}{k-1}}}{\rho^{\frac{1}{k-1}}}.
\]
Thus we have the desired upper bound for $\rho(G)-\rho(G-v)$.

Suppose that  the upper bound for $\rho(G)-\rho(G-v)$ is achieved.
By the above argument, $\mathbf{y}$ is an eigenvector of $G-v$ associated with $\rho(G-v)$, and the entries of $\mathbf{y}$ are all equal. Thus $G-v$ is regular.
Note that
$\left(\mathcal{A}(G-v)\mathbf{x}^{k-1}\right)_w=\rho(G-v)x_w^{k-1}$
for $w\in [n]\setminus\{v\}$.
Then for any $w\in [n]\setminus\{v\}$, we have
\begin{align*}
\rho(G)x^{k-1}_w&=
\left(\mathcal{A}(G)\mathbf{x}^{k-1}\right)_w\\
&=\left(\mathcal{A}(G-v)\mathbf{x}^{k-1}\right)_w
+\sum_{e\in E_w(G)\cap E_v(G)}\mathbf{x}^{e\setminus \{w\}}\\
&=\rho(G-v)x_w^{k-1}+\sum_{e\in E_w(G)\cap E_v(G)}\mathbf{x}^{e\setminus \{w\}}.
\end{align*}
Thus $\sum_{e\in E_w(G)\cap E_v(G)}\mathbf{x}^{e\setminus \{w\}}=\left(\rho(G)-\rho(G-v)\right)x^{k-1}_w>0$.
This implies that  $v$ is adjacent to any vertex of $G$.

Conversely, suppose that $v$ is adjacent to any other vertex, and $G-v$ is regular. As $G$ is linear,
we have $d_G(v)=\frac{n-1}{k-1}$. Assume the degree of any vertex of $G-v$ is $r$.

Let $t$ be the largest positive number  satisfying $t(t-r)^{k-1}=d_G(v)$. Evidently, $t>r$. Let
\[
a=\sqrt[k]{(t-r)^k+n-1}
 \]
We construct a vector $\mathbf{\hat{x}}$ on $V(G)$ with each entry corresponding to any vertex different from $v$ to be $y$ such that  $\hat{x}_v=\frac{t-r}{a}$ and $y=\frac{1}{a}$, where
Then $\hat{x}_v^{k}+(n-1)y^k=1$,   $t\hat{x}_v^{k-1}=d_G(v)y^{k-1}$ and $(t-r)y=\hat{x}_v$. Thus,  $\mathbf{\hat{x}}$ is a positive eigenvector of $\mathcal{A}(G)$ associated with  an H-eigenvalue  $t$. Note that  $\mathcal{A}(G)$ is weakly irreducible  as $G$ is connected.  By Proposition \ref{PF}, $t=\rho(G)$.
Now by Theorem~\ref{up-2},
\[
\rho(G)-\rho(G-v)=\frac{\rho(G)(k-1)\hat{x}_v^k}{1-\hat{x}_v^k}
=\frac{\hat{x}_v}{y}
=\sqrt[k-1]{\frac{d_G(v)}{\rho(G)}}.
\]
That is, the upper bound for $\rho(G)-\rho(G-v)$ is attained.
\end{proof}


For a Steiner system $S(t,k,n)$,  the degree (known as replication number) of any vertex is $\frac{n-1}{k-1}$ and
it possesses  $\frac{n(n-1)}{k(k-1)}$ edges.

\begin{Theorem}
Let $G$ be a connected $k$-uniform linear hypergraph with vertex set $[n]$, where $n\ge k\ge 2$.
Let $\gamma(G)=\max\{\rho(G-v): v\in [n]\}$.
Then
\[
\gamma(G)\ge \rho(G)-\sqrt[k-1]{\frac{\delta(G)}{\rho(G)}}\ge \rho(G)-1
\]
and either lower bound is attained 
 if
and only if $G$ is a Steiner system $S(2, k, n)$.
\end{Theorem}

\begin{proof}
Let $v^*$ be the vertex of $G$ with the minimum degree, $\delta(G)$.
Then $\gamma(G)
\ge \rho(G-v^*)$.
By Lemma \ref{ggv}, we have
\begin{equation}\label{bo}
\rho(G)-\gamma(G)\le\rho(G)-\rho(G-v^*)\le\sqrt[k-1]{\frac{d_G(v^*)}{\rho(G)}}
=\sqrt[k-1]{\frac{\delta(G)}{\rho(G)}}\le 1.
\end{equation}

Suppose that $\rho(G)-\gamma(G)=\sqrt[k-1]{\frac{\delta(G)}{\rho(G)}}$.
Then the first and the second  inequalities in \eqref{bo} are equalities.
By Theorem~\ref{ggv}, $v^*$ is adjacent to any other vertex and $G-v^*$ is regular.
Since $G$ is a linear hypergraph,
the degree of $v^*$ in $G$ is $\frac{n-1}{k-1}$, so $G$ is $\frac{n-1}{k-1}$-regular. 
That is, 
every $2$-element vertex subset is contained in precisely one edge.
Hence, $G$ is a Steiner system  $S(2, k, n)$.

Suppose that $\rho(G)-\gamma(G)=1$.
Then all inequalities in \eqref{bo} are equalities.
So $G$ is $\frac{n-1}{k-1}$-regular. 
As $G$ is a linear hypergraph again,
every $2$-element vertex subset is contained in precisely one edge.
Hence, $G$ is a Steiner system  $S(2, k, n)$.


Conversely, suppose that $G$ is a Steiner system  $S(2, k, n)$. Then
the degree of any vertex of $G$ is $\frac{n-1}{k-1}$. For any vertex $v$,
the degree of any vertex of $G-v$ is $\frac{n-1}{k-1}-1=\frac{n-k}{k-1}$. So
$\rho(G)-\gamma(G)=\sqrt[k-1]{\frac{\delta(G)}{\rho(G)}}=1$.
\end{proof}

The following theorem generalizes the result in \cite{Van} from graphs to hypergraphs.

\begin{Theorem}\label{bound-2}
Let $G$ be a $k$-uniform hypergraph with vertex set $[n]$.
Let $E\subseteq E(G)$. Set ${\bf x}$ and ${\bf y}$ the unit eigenvectors of $\rho(G)$ and $\rho(G-E)$, respectively.
Then
\[
\rho(G)-k\sum_{e\in E}{\bf x}^e\le \rho(G-E)\leq \rho(G)-k\sum_{e\in E}{\bf y}^e.
\]
\end{Theorem}

\begin{proof} We need only to show
\[
k\sum_{e\in E}{\bf y}^e\leq \rho(G)-\rho(G-E) \leq k\sum_{e\in E}{\bf x}^e.
\]

Since $\rho(G)= \mathcal{A}(G){\bf x}^k$,
we have
\begin{align*}
\rho(G-E)\geq&\mathcal{A}(G-E){\bf x}^k\\
&=k\sum_{e\in E(G)-E}{\bf x}^e\\
&=k\sum_{e\in E(G)}{\bf x}^e-k\sum_{e\in E}{\bf x}^e\\
&=\rho(G)-k\sum_{e\in E}{\bf x}^e,
\end{align*}
from which the upper bound for $\rho(G)-\rho(G-E)$ follows.

Since $\rho(G-E)= \mathcal{A}(G-E){\bf y}^k=k\sum_{e\in E(G)-E}{\bf y}^e$,
we have
\begin{align*}
\rho(G) & \geq \mathcal{A}(G){\bf y}^k\\
&=k\sum_{e\in E(G)}{\bf y}^e\\
&=k\sum_{e\in E(G)-E}{\bf y}^e+k\sum_{e\in E}{\bf y}^e\\
&=\rho(G-E)+k\sum_{e\in E}{\bf y}^e,
\end{align*}
from which the lower bound for $\rho(G)-\rho(G-E)$ follows.
\end{proof}

\begin{Example}\label{up-1++}
Let $G$ be a connected  $k$-uniform regular hypergraph with vertex set $[n]$. Let $e\in E(G)$. Note that the unite positive vector corresponding to $\rho(G)$ is $(\frac{1}{\sqrt[k]{n}},\ldots, \frac{1}{\sqrt[k]{n}})^\top$. By Theorem~\ref{bound-2},  we have
\[
 \rho(G-e)\ge \rho(G)-\frac{k}{n}.
\]
\end{Example}

\begin{Example}
Let $n$ be a positive integer at least $3$.
 Let $C_{2n}^3$ be the $3$-uniform hypercycle with vertex set $\{v_i:1\leq i\leq 2n\}$
and edge set $\{\{v_{2i-1},v_{2i},v_{2i+1}\}:1\leq i \leq n\}$, where $v_{2n+1}=v_1$.
It is easily checked that $\rho(C_{2n}^3)=2^\frac{2}{3}$ with unit eigenvector
${\bf x }$ such that $x_v=\sqrt[3]{\frac{2}{3n}}$ if $v=v_{2i-1}$ for each $1\leq i\leq n$ and $x_v=\sqrt[3]{\frac{1}{3n}}$ if $v=v_{2i}$ for each $1\leq i\leq n$.
Let $P_{2n-1}^3$ be the $3$-uniform hyperpath with vertex set $\{v_i:1\leq i\leq 2n-1\}$
and edge set $\{\{v_{2i-1},v_{2i},v_{2i+1}\}:1\leq i \leq n-1\}$.
Then $\rho(P_{2n-1}^3)=2^{\frac{2}{3}}\cos^{\frac{2}{3}}\frac{\pi}{n+1}$.
Obviously, $P_{2n-1}^3\cong C_{2n}^3-v_{2n}$.
In Theorem~\ref{bound-1+} or Example~\ref{up-1++}, the lower bound on
$\rho(P_{2n-1}^3)$ is $\frac{3n-3}{3n-1}\cdot 2^{\frac{2}{3}}$. Obviously, the ratio of the value of $\rho(P_{2n-1}^3)$ and the lower bound given above tends to $1$ when $n\to \infty$.
\end{Example}

\section{Least H-eigenvalues}

In this section we study
least H-eigenvalues of symmetric tensors and uniform hypergraphs.
For a symmetric tensor $\mathcal{T}$ with at least H-eigenvalue, we call a unit  eigenvector of $\mathcal{T}$ associated to $\lambda_{\min}(\mathcal{T})$ a  least eigenvector of $\mathcal{T}$.
In particular, a   least eigenvector of a $k$-uniform hypergraph $G$ is a  least eigenvector of $\mathcal{A}(G)$. The results in \cite{XZ} are generalized to tensors and uniform hypergraphs. We note that $k$ is always even in this section.

\begin{Theorem}\label{bound-3}
Let $\mathcal{T}$ be a zero diagonal symmetric tensor of order $k$ and dimension $n$, where $k$ is even and $n,k\ge 2$.   Let ${\bf x}$ be a least  eigenvector of $\mathcal{T}$.
Let  $\emptyset\ne I\subset [n]$. Then
\begin{align*}
\lambda_{\min}(\mathcal{T}) & \leq \lambda_{\min}(\mathcal{T}[I]) \\
& \leq \lambda_{\min}(\mathcal{T})\left(1-k\sum_{i\in [n]\setminus I}x_i^k\right)-\sum_{j=1}^{k-1}(-1)^j{k \choose j+1}\sum_{i_1,\ldots,i_{j+1}\in [n]\setminus I \atop i_{j+2},\ldots,i_k\in[n]}t_{i_1 \ldots i_k}x_{i_1}\cdots x_{i_k}.
\end{align*}

Moreover, if $\sum_{i\in I}x_i^k\ne 0$, then
\[
\lambda_{\min}(\mathcal{T}[I])\leq \frac{\lambda_{\min}(\mathcal{T})\left(1-k\sum_{i\in [n]\setminus I} x_i^k\right)-\sum_{j=1}^{k-1}(-1)^j{k\choose j+1}\sum_{i_1,\ldots,i_{j+1}\in [n]\setminus I \atop i_{j+2},\ldots, i_k\in [n]}t_{i_1 \ldots  i_k} x_{i_1}\cdots x_{i_k}}{\sum_{i\in I} x_i^k}.
\]
\end{Theorem}

\begin{proof}
Since $\mathcal{T}$ is symmetric, we have by the same argument as in Theorem~\ref{bound-1} that
\begin{align*}
(\mathcal{T}-\mathcal{T}_I){\bf x}^k
&=k\lambda_{\min}(\mathcal{T})\sum_{i\in [n]\setminus I} x_i^k+\sum_{j=1}^{k-1}(-1)^j{k\choose j+1}\sum_{i_1,\ldots,i_{j+1}\in [n]\setminus I \atop i_{j+2},\ldots, i_k\in [n]}t_{i_1 \ldots i_k} x_{i_1}\cdots x_{i_k}.
\end{align*}
Note that  $\lambda_{\min}(\mathcal{T})=\mathcal{T}{\bf x}^k$.
Thus by Lemmas~\ref{S-interl} and~\ref{sym},
\begin{align*}
\lambda_{\min}(\mathcal{T}[I])
&=\lambda_{\min}(\mathcal{T}_I)\\
&\leq \mathcal{T}_I{\bf x}^k\\
&=\mathcal{T}{\bf x}^k- (\mathcal{T}-\mathcal{T}_I){\bf x}^k\\
&=\lambda_{\min}(\mathcal{T})-\left(k\lambda_{\min}(\mathcal{T})\sum_{i\in [n]\setminus I} x_i^k+\sum_{j=1}^{k-1}(-1)^j{k\choose j+1}\sum_{i_1,\ldots,i_{j+1}\in [n]\setminus I \atop i_{j+2},\ldots, i_k\in [n]}t_{i_1 \ldots  i_k} x_{i_1}\cdots x_{i_k}\right)\\
&=\lambda_{\min}(\mathcal{T})\left(1-k\sum_{i\in [n]\setminus I} x_i^k\right)-\sum_{j=1}^{k-1}(-1)^j{k\choose j+1}\sum_{i_1,\ldots,i_{j+1}\in [n]\setminus I \atop i_{j+2},\ldots, i_k\in [n]}t_{i_1 \ldots  i_k} x_{i_1}\cdots x_{i_k}.
\end{align*}

Let ${\bf y}\in \mathbb{R}^{|I|}$ be a unit eigenvector corresponding to $\lambda_{\min}(\mathcal{T}[I])$.
Construct a new unit vector ${\bf z}\in \mathbb{R}^n$ such that $z_i=y_i$ if $i\in I$
and $z_i=0$ otherwise. Then by Lemma~\ref{sym}
\begin{align*}
\lambda_{\min}(\mathcal{T})\leq \mathcal{T}{\bf z}^k
=\mathcal{T}_I{\bf z}^k+ (\mathcal{T}-\mathcal{T}_I){\bf z}^k
=\mathcal{T}[I]{\bf y}^k + 0
=\lambda_{\min}(\mathcal{T}[I])
\end{align*}
as desired. This proves the first part.

Suppose that $\sum_{i\in I}x_i^k\ne 0$.
By the above argument,
we have
\[
\mathcal{T}_I{\bf x}^k
=\lambda_{\min}(\mathcal{T})\left(1-k\sum_{i\in [n]\setminus I} x_i^k\right)-\sum_{j=1}^{k-1}(-1)^j{k\choose j+1}\sum_{i_1,\ldots,i_{j+1}\in [n]\setminus I \atop i_{j+2},\ldots, i_k\in [n]}t_{i_1 \ldots  i_k} x_{i_1}\cdots x_{i_k}.
\]
Let ${\bf w}$ be a vector in $\mathbb{R}^{|I|}$ such that $w_i=x_i$ if $i\in I$. Thus by Lemma~\ref{sym},
\begin{align*}
\lambda_{\min}(\mathcal{T}[I]) &\leq \frac{ \mathcal{T}[I]{\bf w}^k}{\|{\bf w}\|_k^k}
=\frac{\mathcal{T}_I{\bf x}^k}{\sum_{i\in  I} x_i^k}\\
&=\frac{\lambda_{\min}(\mathcal{T})\left(1-k\sum_{i\in [n]\setminus I} x_i^k\right)-\sum_{j=1}^{k-1}(-1)^j{k\choose j+1}\sum_{i_1,\ldots,i_{j+1}\in [n]\setminus I \atop i_{j+2},\ldots, i_k\in [n]}t_{i_1 \ldots  i_k} x_{i_1}\cdots x_{i_k}}{\sum_{i\in I} x_i^k},
\end{align*}
proving the second part.
\end{proof}

\begin{Example}
Let $\mathcal{T}$ be a tensor of order $4$ and dimension $3$, where
$t_{1122}=t_{1212}=t_{1221}=t_{2121}=t_{2211}=t_{2112}=1$,
$t_{1222}=t_{2122}=t_{2212}=t_{2221}=3$,
$t_{3222}=t_{2322}=t_{2232}=t_{2223}=3$,
and otherwise,
$t_{ijst}=0$.
Let $I=\{2,3\}$. By MATLAB, we have $\lambda_{\min}(\mathcal{T})=-9.9307$ with eigenvector ${\bf x}_0=(-0.5239,1,-0.671)^\top$ and $\lambda_{\min}(\mathcal{T}[I])=-6.8385$.
Let ${\bf x}=\frac{{\bf x}_0}{\|{\bf x}_0\|_4}$.
The first upper bound for $\rho(\mathcal{T}[I])$ in Theorem~\ref{bound-3}
is equal to
\begin{align*}
&\quad \lambda_{\min}(\mathcal{T})\left(1-k\sum_{i\in [n]\setminus I}x_i^k\right)-\sum_{j=1}^{k-1}(-1)^j{k \choose j+1}\sum_{i_1,\ldots,i_{j+1}\in [n]\setminus I \atop i_{j+2},\ldots,i_k\in[n]}t_{i_1 \ldots i_k}x_{i_1}\cdots x_{i_k}\\
&=\lambda_{\min}(\mathcal{T})\left(1-4x_1^4\right)-(-1)^1{4 \choose 2}t_{1122}x_1^2x_2^2\\
&=-9.9307\times \left(1-\frac{4\times (-0.5239)^4}{(-0.5239)^4+1+(-0.671)^4} \right)-(-1)^1{4 \choose 2}\times\frac{1\times (-0.5239)^2\times 1^2}{(-0.5239)^4+1+(-0.671)^4}\\
&=-6.3007.
\end{align*}
and the second upper bound in Theorem~\ref{bound-3} is equal to
\begin{align*}
\frac{\lambda_{\min}(\mathcal{T})\left(1-4x_1^4\right)-(-1)^1{4 \choose 2}t_{1122}x_1^2x_2^2}{ x_2^k+x_3^k}=\frac{-6.3007}{\frac{1+(-0.671)^4}{(-0.5239)^4+1+(-0.671)^4}}
= -6.6954.
\end{align*}
\end{Example}


By similar argument as in Corollary~\ref{up-1}, we have

\begin{Corollary}\label{lower-1}
Let $G$ be a $k$-uniform hypergraph with vertex set $[n]$,  where $k$ is even, and $n,k\ge 2$
Let ${\bf x}$ be a least eigenvector of $\mathcal{A}(G)$.
If $\emptyset\ne I\subset [n]$,
then
\[
\lambda (G)\le \lambda (G-I)\le \lambda (G)\left(1-k\sum_{i\in I}x_i^k\right)+k\sum_{j=2}^{k}\sum_{e:|e\cap I|=j}(j-1){\bf x}^e.
\]

Moreover, if $\sum_{i\in I}x_i^k\ne 0$, then
\[
\lambda (G-I)\le \frac{\lambda (G)\left(1-k\sum_{i\in I}x_i^k\right)+k\sum_{j=2}^{k}\sum_{e:|e\cap I|=j}(j-1){\bf x}^e}{\sum_{i\in I}x_i^k}.
\]
\end{Corollary}

\begin{Example}
Let $G$ be a $4$-uniform hypergraph with vertex set $[6]$, where
$E(G)$$=\{\{1,2,3,4\}$, $\{3,4,5,6\}$,$\{1,3,4,5\}\}$.
Let $I=\{5,6\}$. Obviously, $\lambda(G-I)=-1$.
By MATLAB, $\lambda(G)=-2.1908$ with eigenvector ${\bf x}_0=(-0.9112,0.7465,1,1,0.9112,-0.7465)^\top$.
Let ${\bf x}=\frac{{\bf x}_0}{\|{\bf x}_0\|_4}$.
The first upper bound for $\lambda(G-I)$ in above  corollary
is
\begin{align*}
&\quad \lambda(G)\left(1-4\sum_{i\in I}x_i^4\right)+4\sum_{j=2}^{4}\sum_{e:|e\cap I|=j}(j-1){\bf x}^e\\
&=\lambda(G)\left(1-4(x_5^4+x_6^4)\right)+4(2-1)x_3x_4x_5x_6\\
&=-2.1908 \times \left(1-\frac{4\times (0.9112^4+(-0.7465)^4)}{(-0.9112)^4+ 0.7465^4+1+1+0.9112^4+(-0.7465)^4}\right)\\
&\quad +4\times\frac{1\times 1\times 0.9112\times (-0.7465)}{(-0.9112)^4+ 0.7465^4+1+1+0.9112^4+(-0.7465)^4}\\
&=-0.6803,
\end{align*}
and the second one is equal to
\[
\frac{\lambda(G)\left(1-4(x_5^4+x_6^4)\right)+4(2-1)x_3x_4x_5x_6}{x_5^4+x_6^4}
=-0.9071.
\]
\end{Example}

\begin{Theorem}\label{bound-4}
Let $G$ be a $k$-uniform hypergraph with vertex set $[n]$, where $k$ is even and $n,k\ge 2$. Let $E\subseteq E(G)$.
Let ${\bf x}$ and ${\bf y}$ be the least eigenvectors  of $G$ and $G-E$, respectively.
Then
\[
\lambda (G)-k\sum_{e\in E}{\bf y}^e\le \lambda (G-E)\le \lambda (G)-k\sum_{e\in E}{\bf x}^e.
\]
\end{Theorem}

\begin{proof} It suffices to show that
\[
k\sum_{e\in E}{\bf x}^e\le \lambda (G)- \lambda (G-E)\le k\sum_{e\in E}{\bf y}^e.
\]

Since $\lambda (G)= \mathcal{A}(G){\bf x}^k=k\sum_{e\in E(G)}{\bf x}^e$,
we have
\begin{align*}
 \lambda (G-E)& \leq \mathcal{A}(G-E){\bf x}^k
=k\sum_{e\in E(G)-E}{\bf x}^e
=k\sum_{e\in E(G)}{\bf x}^e-k\sum_{e\in E}{\bf x}^e
&= \lambda (G)-k\sum_{e\in E}{\bf x}^e,
\end{align*}
and so $ \lambda (G)- \lambda (G-E)\geq k\sum_{e\in E}{\bf x}^e$.
On the other hand,
since $\lambda (G-E)= \mathcal{A}(G-E){\bf y}^k=k\sum_{e\in E(G)-E}{\bf y}^e$,
we have
\begin{align*}
\lambda (G)\leq \mathcal{A}(G){\bf y}^k
=k\sum_{e\in E(G)}{\bf y}^e
=k\sum_{e\in E(G)-E}{\bf y}^e+k\sum_{e\in E}{\bf y}^e
&=\lambda (G-E)+k\sum_{e\in E}{\bf y}^e,
\end{align*}
and so $ \lambda (G)-\lambda (G-E)\leq k\sum_{e\in E}{\bf y}^e$.
\end{proof}

\begin{Example}
Let $G$ be a $4$-uniform hypergraph with vertex set $[6]$, where
$E(G)$$=\{\{1,2,3,4\}$, $\{3,4,5,6\}$,$\{1,3,4,5\},\{1,2,4,5\}\}$.
Let $E=\{\{1,2,3,4\}\}$.
By MATLAB, $\lambda(G)=-2.8786$ with eigenvector ${\bf x}_0=(-0.9457,0.848,0.928,1,0.928,-0.6688)^\top$
and $\lambda(G-E)=-2.1908$ with eigenvector ${\bf y}_0=(-0.9112,0.7465,0.9112,1,1,-0.7465)^\top$.
Let ${\bf x}=\frac{{\bf x}_0}{\|{\bf x}_0\|_4}$ and ${\bf y}=\frac{{\bf y}_0}{\|{\bf y}_0\|_4}$.
The lower bound  for $\lambda(G-E)$ in Theorem~\ref{bound-4}
is \begin{align*}
&\quad \lambda (G)-k\sum_{e\in E}{\bf y}^e\\
&=\lambda (G)-4y_1y_2y_3y_4\\
&=-2.8786-4\times  \frac{-0.9112\times 0.7465\times 0.9112\times 1}{(-0.9112)^4+0.7465^4+0.9112^4+1+1+(-0.7465)^4}\\
&=-2.2587,
\end{align*}
and the upper bound  for $\lambda(G-E)$ in Theorem~\ref{bound-4}
is \begin{align*}
&\quad \lambda (G)-k\sum_{e\in E}{\bf x}^e\\
&=\lambda (G)-4x_1x_2x_3x_4\\
&=-2.8786-4\times \frac{(-0.9457)\times 0.848\times 0.928\times 1}{(-0.9457)^4+0.848^4+0.928^4+1+0.928^4+(-0.6688)^4}\\
&=-1.4467.
\end{align*}
\end{Example}

\begin{Theorem}\label{linear}
Let $G$ be a linear $k$-uniform hypergraph with at least one edge, where $k$ is even and $n, k\geq 2$. Let ${\bf x}$ be a least eigenvector of $G$.
Then for $i\in V(G)$,
\begin{align}
x_i^k\leq \frac{d_i}{d_i+(k-1)\lambda(G)^\frac{k}{k-1}}\label{Eq-22-5-25}
\end{align}
with equality if and only if for $j\in V(G)\setminus \{i\}$, $x_j^k= \frac{\lambda(G)^{\frac{k}{k-1}}x_i^k}{d_i^2}$ if $j\sim i$
and $x_j=0$ otherwise, and the sign of ${\bf x}^{e\setminus\{i\}}$ for each $e\in E_i(G)$ is the same.
\end{Theorem}
\begin{proof}
Let $\lambda=\lambda(G)$.
From the eigenequation of $G$ at $i$,
\begin{align}
\lambda^{\frac{k}{k-1}}x_i^k&=\left(\sum_{e\in E_i(G)}{\bf x}^{e\setminus\{i\}}\right)^\frac{k}{k-1} \notag\\
& \leq \left(\sum_{e\in E_i(G)}\left|{\bf x}^{e\setminus\{i\}}\right|\right)^\frac{k}{k-1}  \notag\\
&\leq  \left(d_i\left(\frac{\sum_{\{i,i_2,\ldots,i_k\}\in E_i(G)}x_{i_2}^k\cdots x_{i_k}^k}{d_i}\right)^\frac{1}{k}\right)^\frac{k}{k-1}  \notag\\
&=d_i\left(\sum_{\{i,i_2,\ldots, i_k\}\in E_i(G)}x_{i_2}^k\cdots x_{i_k}^k\right)^\frac{1}{k-1}  \label{Eq-22-5-25+}\\
&\leq d_i \left(\sum_{\{i,i_2,\ldots, i_k\}\in E_i(G)}\left(\frac{x_{i_2}^k+\cdots + x_{i_k}^k}{k-1}\right)^{k-1}\right)^\frac{1}{k-1}   \notag\\
& \leq \frac{d_i}{k-1}\sum_{\{i,i_2,\ldots, i_k\}\in E_i(G)}\left(x_{i_2}^k+\cdots + x_{i_k}^k\right) \notag\\
&= \frac{d_i}{k-1} \sum_{j:j\sim i} x_j^k    \notag\\
& \leq  \frac{d_i}{k-1} \left(1-x_i^k\right), \notag
\end{align}
where the first and the last two inequalities follows trivially, the second and the third inequalities  follow respectively  from the power mean inequality and the arithmetic-geometric mean inequality.
So \eqref{Eq-22-5-25} follows.

Suppose that equality holds in \eqref{Eq-22-5-25}. Then all inequalities in \eqref{Eq-22-5-25+} are equalities. From the first inequality,
we see that the sign of ${\bf x}^{e\setminus\{i\}}$ for any $e\in E_i(G)$ is the same.
From the third inequality, we know that for each $\{i,i_2,\ldots,i_k\}\in E_i(G)$, $x_{i_2}^k=\cdots=x_{i_k}^k$. From the last inequality, we find that
either $j\sim i$ for all $j\in V(G)\setminus \{i\}$ or $x_j^k=0$ for each $j\nsim i$.
So $x_j^k= \frac{\lambda^\frac{k}{k-1}x_i^k}{d_i^2}$ if $j\sim i$ and $0$ otherwise.

Conversely, suppose that $j\in V(G)\setminus \{i\}$, $x_j^k= \frac{\lambda(G)^{\frac{k}{k-1}}x_i^k}{d_i^2}$ if $j\sim i$
and $x_j=0$ otherwise, and the sign of ${\bf x}^{e\setminus\{i\}}$ for each $e\in E_i(G)$ is the same. Then all  inequalities in \eqref{Eq-22-5-25+} are equalities, so  \eqref{Eq-22-5-25} is an equality.
\end{proof}

Suppose that $G$ is a $k$-uniform hypergraph with at least one edge, where $k$ is even and $k\geq 2$.
Let $c_{\max}$ be the largest component among all least eigenvectors of $\mathcal{A}(G)$.

\begin{Theorem}\label{linear+}
Let $G$ be a linear $k$-uniform hypergraph with vertex set $[n]$ and at least one edge, where $k$ is even and $n, k\geq 2$.
Then
\[c_{\max}\leq \sqrt[k]{\frac{n-1}{n-1+(k-1)^2}}\]
with equality if and only if $x_j^k=\frac{x_i^k}{\left(\frac{n-1}{k-1}\right)^\frac{k}{k-1}}$ if $j\sim i$ and $x_i=0$ otherwise, $\lambda(G)=-1$, maximum degree is $\frac{n-1}{k-1}$,  and the sign of ${\bf x}^{e\setminus{i}}$ for each $e\in E_i(G)$ is the same.
\end{Theorem}

\begin{proof}
Let ${\bf x}$ be a least eigenvector of $G$ containing $c_{\max}$.
Suppose without loss of generality that $c_{\max}=x_1$.
As $G$ is linear, we have  $d_i\leq \frac{n-1}{k-1}$ for $i\in V(G)$.
As  $k$ is even and $|E(G)|\ge 1$,
we have $\lambda(G)\leq -1$ by Corollary~\ref{lower-1} and the fact that the least H-eigenvalue of the $k$-uniform hypergraph consisting of exactly one edge is $-1$.
So, by Theorem~\ref{linear},
\[x_i^k\leq \frac{d_i}{d_i+(k-1)\lambda(G)^{\frac{k}{k-1}}}
\leq \frac{\frac{n-1}{k-1}}{\frac{n-1}{k-1}+(k-1)(-1)^{\frac{k}{k-1}}}
= \frac{n-1}{n-1+(k-1)^2}\]
with equalities if and only if $x_j^k=\frac{x_i^k}{\left(\frac{n-1}{k-1}\right)^\frac{k}{k-1}}$ if $j\sim i$ and $x_i=0$ otherwise, $\lambda(G)=-1$, $d_1=\frac{n-1}{k-1}$, and the sign of ${\bf x}^{e\setminus{i}}$ for each $e\in E_i(G)$ is the same.
Thus the result follows.
\end{proof}

For even integer $k$, we say a $k$-uniform hypergraph $G$ is odd-bipartite
if $V(G)$ can be partitioned into two disjoint vertex set $V_1$ and $V_2$
such that each edge intersects each of $\{V_1,V_2\}$ with odd number of
vertices.

\begin{Theorem}\label{odd}
Let $G$ be an odd-bipartite, connected $k$-uniform hypergraph with $m$ edges, where $k$ is even, and $m\geq 1$.
Then $c_{\max}\geq \sqrt[k]{-\frac{\lambda(G)}{km}}$
with equality if and only if $G$ is regular.
\end{Theorem}

\begin{proof}
Let ${\bf x}$ be the least eigenvector of $G$ containing $c_{\max}$.
Since $G$ is  odd-bipartite, $\lambda(G)=-\rho(G)$ \cite{SSW}.
Let ${\bf \widetilde{x}}$ be the vector such that  $\widetilde{x}_i=x_i$ for each $i\in V(G)$. Obviously, ${\bf \widetilde{x}}$ is unit.
For any $u\in V(G)$,
\[
\rho(G)\widetilde{x}_u^{k-1}=\rho(G)|x_u|^{k-1}=\left|\sum_{e\in E_u(G)} {\bf x}^{e\setminus \{u\}}\right|\le
\sum_{e\in E_u(G)} \left|{\bf x}^{e\setminus \{u\}}\right|=\sum_{e\in E_u(G)} {\bf \widetilde{x}}^{e\setminus \{u\}},
\]
i.e., \[\rho(G)\widetilde{x}_u^{k-1}\leq \sum_{e\in E_u(G)} {\bf \widetilde{x}}^{e\setminus \{u\}}.\]
As ${\bf \widetilde{x}}$ is nonnegative,
\[\rho(G)\widetilde{x}_u^{k}\leq \sum_{e\in E_u(G)} {\bf \widetilde{x}}^{e}.\]
By Lemma~\ref{RQ}, we have
\[
 \rho(G)=\sum_{u\in V(G)}\rho(G)\widetilde{x}_u^k\leq \sum_{u\in V(G)} \sum_{e\in E_u(G)} {\bf \widetilde{x}}^{e}=\sum_{e\in E(G)} k{\bf \widetilde{x}}^{e}\leq \rho(G),
\]
so $\rho(G)=\sum_{e\in E(G)} k{\bf \widetilde{x}}^{e}$ and
$\rho(G) \widetilde{x}_u^{k-1}=\sum_{e\in E_u(G)} {\bf \widetilde{x}}^{e\setminus \{u\}}$.
That is,  ${\bf \widetilde{x}}$ is a unit nonnegative eigenvector corresponding to $\rho(G)$.
As $G$ is connected, ${\bf \widetilde{x}}$ is positive.
It follows that
\[
-\lambda(G)=\rho(G)=k\sum_{e\in E(G)}{\bf \widetilde{x}}^e\leq kmc_{\max}^k
\]
with equality if and only if $|x_i|=c_{\max}$, i.e., $G$ is regular.
\end{proof}


\vspace{5mm}

\noindent {\bf Acknowledgements.} We thank Prof. Tan Zhang for kind discussions.
This work was supported by the National Natural Science Foundation of China (Nos.~12071158 and 11801410).

\end{document}